\documentclass[11pt]{article}
\usepackage{braket,amsfonts}

\usepackage{array}

\usepackage{algorithmic}

\usepackage{graphicx,epstopdf}

\usepackage{xspace}
\usepackage{bold-extra}
\usepackage{bm}
\usepackage{mathrsfs,bbm}
\usepackage{amsmath,amsthm}
\usepackage{float}
\usepackage{color}
\usepackage{multirow}
\usepackage{appendix}

\usepackage{enumerate}

\newcommand{\Ht}{H^{m}(\Omega_h)}
\newcommand{\Cnurho}{C_{\nu,\rho}}
\newcommand{\norm}   [1] {\Vert#1\Vert}
\newcommand{\normb}  [1] {|\![#1]\!|}

\newcommand{\jump}   [1] {[\![#1]\!]}
\newcommand{\Hcal} {\mathcal{H}}
\newcommand{\Gcal} {\mathcal{G}}

\newcommand{\jbhalf}{j-\frac12}
\newcommand{\jfhalf}{j+\frac12}

\newcommand{\sumn}{\sum_{j=1}^N}

\allowdisplaybreaks[4]

\newtheorem{lemma}{Lemma}[section]
\newtheorem{theorem}{Theorem}[section]

\newtheorem{remark}{Remark}[section]
\newtheorem{exmp}{Example}[section]

\numberwithin{equation}{section}
\newtheorem{crl}{Corollary}[section]

\newcommand{\jl}{j-\frac 12}
\newcommand{\jr}{j+\frac 12}

\allowdisplaybreaks[4]

\title{A Local Discontinuous Galerkin method for the modified Camassa--Holm equation\footnotemark[1]}
\author{
Xiang-Ke Chang\footnotemark[2]
\and
Yong Liu\footnotemark[3]
\and
Qi Tao\footnotemark[4]
}

\date{}

\begin{document}
\maketitle

\renewcommand{\thefootnote}{\fnsymbol{footnote}}
\footnotetext[1]{X.-K. Chang's research is partially supported by Beijing Natural
Science Foundation under the grant No. JQ26004, NSFC grant 12571270, 12288201, and the Youth Innovation Promotion Association (CAS).
Y. Liu's research is partially supported by NSFC grant 12571395, 12288201, the Strategic Priority Research Program of the Chinese Academy of Sciences under the Grant No. XDB0640000, and the Youth Innovation Promotion Association (CAS). Q. Tao's research is partially supported by Beijing Natural
Science Foundation under the grant No. JQ26004, NSFC grant 12301464, and Young Elite Scientists Sponsorship Program of the Beijing High Innovation Plan. }
\footnotetext[2]{ICMSEC, State Key Laboratory of Mathematical Sciences (SKLMS), Academy of Mathematics and Systems Science, and School of Mathematical Science, University of Chinese Academy of Sciences, Chinese Academy of Sciences, Beijing 100190, P.~R.~China. E-mail: changxk@lsec.cc.ac.cn}
\footnotetext[3]{ICMSEC, State Key Laboratory of Mathematical Sciences (SKLMS), Academy of Mathematics and Systems Science, and School of Mathematical Science, University of Chinese Academy of Sciences, Chinese Academy of Sciences, Beijing 100190, P.~R.~China.  E-mail: yongliu@lsec.cc.ac.cn}
\footnotetext[4]{Corresponding author. School of Mathematics, Statistics and Mechanics, Beijing University of Technology, Beijing 100124, P.~R.~China. E-mail: taoqi@bjut.edu.cn}

\begin{center}
\small
\begin{minipage}{0.9\textwidth}
\textbf{Abstract.}
In this paper, we propose a local discontinuous Galerkin (LDG) method for the modified Camassa--Holm (mCH) equation that contains cubic nonlinearity and high-order derivative terms. Energy conservation and stability are obtained using the conservative and dissipative fluxes, respectively. The error estimate of the semi-discrete scheme is also established. Numerical examples verify that our theoretical findings are sharp and the proposed LDG method is capable of capturing the peakon solution of the mCH equation.
\medskip

\medskip

\end{minipage}
\end{center}
\setlength{\parindent}{2em}

\section{Introduction}
\label{sec:1}
The present study develops a rigorous high-order numerical framework for the modified Camassa-Holm (mCH) equation, with particular emphasis on the construction, stability properties, and error estimates of a local discontinuous Galerkin discretization. The model equation is written as
\begin{subequations}\label{eq:mch}
\begin{align}
&M_t+2\kappa^2U_x+((U^2-U_x^2)M)_x=0,\label{eq:mch:1}\\
&M-U+U_{xx}=0,\label{eq:mch:2}
\end{align}
\end{subequations}
which serves as a cubic-nonlinear integrable variant of the classical Camassa–Holm (CH) equation. Originally proposed as a formal integrable system by Fokas and Fuchssteiner \cite{Fokas1995PD,fofu1981,fuchssteiner1996some} and independently by Olver and Rosenau in 1996, the equation was later rediscovered by Qiao \cite{Qiao2006JMP} in 2006. Therefore, it is sometimes also called the FORQ equation. It describes the unidirectional propagation of shallow-water waves of moderate amplitude over a flat bottom, wherein $U$ denotes the horizontal velocity at a specific depth and $M$ represents the momentum density (see Chen et al. \cite{chen2022JMFM}). The constant $\kappa$ is a linear dispersion parameter related to the critical shallow-water speed.

From the viewpoint of nonlinear wave theory, the mCH equation belongs to the category of peakon-bearing equations \cite{lundmark2022view}, which have attracted considerable attention over the past three decades due to several remarkable properties. Like its CH counterpart, the mCH equation admits peakon solutions \cite{Chang2018CMP,GuiCIMP2013,QuCMP2013} and enjoys infinitely many conservation laws, including the $H^1$-norm
\begin{align*}
E(U, U_x)=\int_{\Omega}U^2+U_x^2 \,dx.
\end{align*}
Meanwhile, it has been observed that the mCH  equation displays distinctive features from the CH equation.
Its single peakon and periodic peakon solutions are found in \cite{GuiCIMP2013,QuCMP2013}.
 Multipeakon solutions, constructed via the inverse spectral method for the corresponding peakon ODE system, are examined in \cite{Chang2018CMP}, and the Hamiltonian structure and Liouville integrability of such multipeakon systems are addressed in \cite{chang2016lax,chang2017liouville}; see also \cite{Anco2018DCDSS,chang2020isospectral, chang2026mch}. Single soliton and multisoliton solutions are investigated in \cite{MatsunoJMP2013,MatsunoJPA2014,niu2025darboux,Wang2020JPA} etc., as well as the B\"acklund and Darboux transformations and the associated nonlinear superposition formulae. In \cite{sheng2022mch,sheng2024mch}, integrable discretizations of the mCH equation are studied on the basis of multisoliton solutions and bilinear equations.
 Furthermore, analytic aspects from the point of view of PDE theory are established in \cite{chen2015AdvM,Gao2018DCDSS,GuiCIMP2013,li2024orbital,QuCMP2013,yang2005global} etc., including orbital stability of peakons or solitons, well-posedness of both classical and global weak solutions, and wave-breaking phenomena, etc.
It is worth noting that all of these problems are respectively dealt with according to whether  the linear dispersion term $\kappa$ is zero or nonzero. On the other hand, it seems to us that general numerical aspects of the mCH equation with rigorous numerical analysis have never been reported.

 The numerical approach adopted here is based on the discontinuous Galerkin (DG) finite element framework, whose essential feature is the use of locally independent polynomial representations on neighboring cells. Originally devised for first-order partial differential equations—notably nonlinear conservation laws—it has since demonstrated considerable success in that setting \cite{Cockburn1989, Cockburn1998}. Among its notable advantages are a highly localized data structure, excellent scalability for parallel computing, and the flexibility to accommodate arbitrary triangulations, including those with hanging nodes.

For equations involving higher-order spatial derivatives, the local discontinuous Galerkin
(LDG) methodology provides a natural extension of the DG philosophy by introducing auxiliary variables and transforming the original problem into a coupled first-order system. Its underlying strategy consists of recasting the original high-order problem as an equivalent first-order system, to which the standard DG procedure is then applied. The LDG approach was first introduced by Cockburn and Shu for convection–diffusion problems \cite{Cockburn1998local}, and subsequently extended to a variety of models, including the KdV equation \cite{Yan2002}, Burgers–Poisson equations \cite{liu2015}, the Zakharov system \cite{Wang2024,xia2010JCP}, the CH equation \cite{Xu2008CH}, the Degasperis–Procesi (DP) equation \cite{Xu2011DP}, the $\mu$-CH equation \cite{Lu2023JCAM}, and the Novikov equation \cite{Tao2025MC}, among others. For a comprehensive overview of the LDG method, we refer the reader to the review article \cite{Xureview}.

In this paper, we propose a unified LDG method for solving the mCH equation with a free real constant parameter $\kappa$ (i.e. without specifying whether it is zero or nonzero). Energy conservation and stability of the semi-discrete LDG method are established by choosing central fluxes and alternating fluxes, respectively. It is noted that the rigorous energy boundedness
of a fully discrete scheme for such nonlinear equations is beyond the scope of this
paper, and this aspect will be left for future work. Although we do not discuss the energy boundedness of fully discrete schemes in this paper, one can adopt the relaxation Runge--Kutta (RRK) method \cite{Ketchson2019SINUM} to achieve the conservation or dissipation numerically for the fully discrete energy. Here we refer to \cite{Xu2008CH, Xu2011DP, ZhangCH} for the LDG method for the CH and DP equations using explicit RK methods to ensure computational efficiency. 

Obtaining error estimates for the LDG method applied to nonlinear wave equations involving higher-order derivatives poses substantial difficulties. In their work \cite{Xu2008CH}, Xu and Shu established $L^2$-norm error estimates of order $k$ (where $k$ denotes the maximum polynomial degree in the finite element space) for the LDG approximation to the CH equation, under the a priori assumption
\[
\|U-u\|\leq h,\quad \text{for sufficiently small }h,
\]
where $u$ is the numerical solution. Such an assumption has become a standard tool in handling error analyses for DG methods applied to nonlinear problems, as seen in \cite{Xu2011DP, ZhangCH}.

Our analysis proceeds by first introducing suitable projections and deriving the associated projection errors. The primary task then reduces to bounding the discrepancy between the projection and the numerical solution. To accomplish this, we exploit the polynomial structure of the nonlinearity inherent in the mCH equation, which allows us to split the nonlinear error into two contributions: one involving the projection error relative to the exact solution, and the other involving the difference between the projection and the numerical solution. As both the projection and the numerical solution reside in the same finite element space, we are able to control these terms using the nonlinear stability property and the error energy equation.

Because the discrete $H^1$-norm of $u$ can be controlled by the auxiliary variable $r$ (which approximates $U_x$ in the LDG scheme), we are able to derive error estimates for the nonlinear term of $u$ without imposing any a priori assumption on $u$ itself. Nevertheless, to handle the nonlinear terms involving the auxiliary variable $r$, we still need an a priori bound on $r$, specifically
\[
\|R-r\| \leq h^{\frac 12}, \quad \text{for sufficiently small }h.
\]
Because we lack an $L^\infty$-bound for $r$, the a priori hypothesis is necessary to manage the nonlinearity in $r$. This assumption is justifiable via a standard continuation argument provided that $k > \frac12$. Finally, numerical experiments corroborate the sharpness of the predicted $k$-th order convergence.

The remainder of this paper is structured as follows. In Section~\ref{sec2}, we present the LDG discretization for the mCH equation. Section~\ref{sec3} is devoted to establishing both the conservation properties and the energy boundedness of the proposed scheme. The optimal error analysis for the semi-discrete formulation is carried out in Section~\ref{sec4}. Numerical experiments are provided in Section~\ref{sec5}, which confirm our theoretical results and illustrate the scheme's ability to effectively capture the one-peakon and periodic peakon solutions. Finally, we offer some concluding remarks in Section~\ref{sec6}.

\section{The LDG scheme for the mCH equation} \label{sec2}
\setcounter{equation}{0}
In this section, we introduce the semi-discrete LDG scheme for solving the mCH equation (\ref{eq:mch}). First of all, we give some notations to define the LDG scheme.
\subsection{Notations}\label{sec2:0}
Let $\Omega=(x_l, x_r)$ be our computational domain,
$\Omega_h=\{I_j=(x_{\jbhalf},x_{\jfhalf})\}_{j=1}^N$ be the partition of $\Omega$, where $x_{\frac12}=x_l$ and $x_{N+\frac12}=x_r$. Denote the cell length by $h_j=x_{\jfhalf}-x_{\jbhalf}$ for $j=1,\ldots ,N$, and $h=\max\limits_j h_j$. In this paper, we assume $\Omega_h$ is quasi-uniform, i.e., there exists a positive constant $\rho$, 
such that for all $j$ there holds $h_j/h \ge \rho$  as $h$ tends to zero.
	
Associated with the partition $\Omega_h$, we define the discontinuous finite element space
\begin{align}
V_h^k =\big\{\,v\in L^2(\Omega):v|_{I_j}\in \mathcal{P}_k(I_j),\,\forall j=1,\ldots,N\,\big\}\, ,
\end{align}
where $\mathcal{P}_k(I_j)$ denotes the space of polynomials in $I_j$ of degree at most $k\geq 0$. {We define the broken Sobolev space, for $m\geq 1$, $$\Ht=\big\{w\in L^2(\Omega): w|_{I_j}\in H^m(I_j),\,\forall j=1,\ldots,N\,\big\}.$$}It is not hard to see $V_h^k\subset \Ht$. It is allowed to have discontinuities across element interfaces, so we define $v^{\pm}_{\jr}=\lim\limits_{\epsilon\rightarrow 0^+}v(x_{\jr}\pm\epsilon)$ and denote its jump as  $\jump{v}_{\jr} = v^{+}_{\jr} - v^{-}_{\jr}$ and its average as $\{v\}_{\jr} = \frac 12(v^{+}_{\jr} + v^{-}_{\jr})$. Furthermore, we denote
\begin{align*}
&(w,v)_j=\int_{I_j} w(x)v(x) dx, \quad \|v\|_j=\|v\|_{L^2(I_j)},\quad \|v\|=\|v\|_{L^2(\Omega)},\quad \normb{v}^2=\sumn \jump{v}_{\jbhalf}^2,\\
& \|v\|^2_{\partial I_j}=v(x_{\jr}^-)^2+v(x_{\jl}^+)^2,\quad \|v\|^2_{\partial \Omega_h}=\sum_{j=1}^N\|v\|^2_{\partial I_j}, \quad \norm{v}_{\infty} = \norm{v}_{L^\infty(\Omega)}.
\end{align*}

	\subsection{The LDG scheme}
	\label{sec2:1}
   Following the framework of LDG methods, we introduce three auxiliary variables $R=U_x$, $P=(U^2R)_x$, $S=(\frac 13 R^3)_x$, and we rewrite  \eqref{eq:mch} into the following equivalent form
   \begin{subequations}\label{eq:ldg}
   	\begin{align}
   		M_t + f(U)_x-P_x+(R^2U)_x+S_x =&\,0, \\
   		P-(U^2R)_x=&\, 0,\\
   		R-U_x=&\,0,\\
		M-U+R_x=&\,0,\\
		S-g(R)_x=&\,0,
   	\end{align}
   \end{subequations}
with the initial condition
\begin{align}
    U(x,0)=U_0(x),
\end{align}
and periodic boundary conditions. Here $f(U)=2\kappa^2 U+U^3$ and $g(R)=\frac 13 R^3$.
The LDG scheme is defined as follows (we omit the subscript $h$ in the numerical solution to simplify notation):
Let $u(\cdot,0)\in V_h^k$ be an approximation of the initial data $U_0(x)$, and for any $t\in(0,T]$ we find $u(\cdot,t)$, $p(\cdot,t)$, $r(\cdot,t)$, $m(\cdot,t)$, and $s(\cdot,t)$ $\in V_h^k$, such that for each cell $I_j$ and any test functions $v, q, \psi, \varphi, w \in V_h^k$ satisfying
		\begin{subequations}\label{eq:scheme:original}
		\begin{align}
			(m_{t}, v)_j -(f(u),v_x)_j +\hat{f}_{\jr}v^{-}_{\jr}-\hat{f}_{\jl}v^{+}_{\jl}+(p, v_x)_{j}-\hat{p}_{\jr}v^-_{\jr}+\hat{p}_{\jl}v^+_{\jl}\nonumber \\
		\quad -(r^2u, v_x)_j\!+\!\widehat{(r^2u)}_{\jr}v^-_{\jr}\!-\!\widehat{(r^2u)}_{\jl}v^+_{\jl}
		 -(s , v_x)_j\!+\!\widehat{s}_{\jr}v^-_{\jr}\!-\!\widehat{s}_{\jl}v^+_{\jl}= &\,0,\label{LDG:1:local}	
			\\
			(p, q)_j+(u^2r, q_x)_j -\widehat{(u^2r)}_{\jr}q^-_{\jr}+\widehat{(u^2r)}_{\jl}q^+_{\jl}= &\, 0,\label{LDG:2:local}	
			\\
			(r,\psi)_j+(u,\psi_x)_j -\hat{u}_{\jfhalf} \psi_{\jfhalf}^-+ \hat{u}_{\jbhalf} \psi_{\jbhalf}^+  = &\, 0, \label{LDG:3:local}	\\
			(m,\varphi)_j-(u, \varphi)_j-(r,\varphi_x)_j +\hat{r}_{\jfhalf} \varphi_{\jfhalf}^--\hat{r}_{\jbhalf} \varphi_{\jbhalf}^+  = &\, 0, \label{LDG:4:local}\\
			(s,w)_j +(g(r), w_x)_j -\hat{g}_{\jr}w^-_{\jr}+\hat{g}_{\jl}w^+_{\jl}= &\, 0, \label{LDG:5:local}		
		\end{align}
	\end{subequations}
where
\begin{itemize}
  \item $\hat{f},\hat{g}, \,\hat{p}, \,\widehat{(r^2u)}, \,\widehat{(u^2r)}, \,\hat{u}$, $\hat{r}$, and $\hat{s}$ are numerical fluxes. We choose
 \begin{subequations}
              \begin{align}\label{eq:flux:alter}
		\hat{p}=p^+,\quad \widehat{(r^2u)}=\{ur\}r^+,\quad \widehat{(u^2r)}=\{ur\}u^-,\quad \hat{u}=u^-, \quad \hat{r}=r^+,\quad \hat{s}=s^+.
		\end{align}
For $\hat{f}$ and $\hat{g}$, we can choose the following different fluxes such that the numerical scheme is conservative or dissipative:
 \begin{enumerate}[(i)]
  \item  For a dissipative scheme, we choose the monotone flux. Since $f^{\prime}(u)=2\kappa^2+3u^2\geq0$ and $g^{\prime}(r)=r^2\geq0$, we take the upwind fluxes for $f$ and $g$ 		
      \begin{align}
		\hat{f}(u^+, u^-)=f(u^-), \quad  \hat{g}(r^+, r^-)=g(r^+).\label{eq:flux:fd}
		\end{align}
  \item  For a conservative scheme, we take the central flux for both $f$ and $g$
  \begin{align}
		&\hat{f}(u^+, u^-)=\kappa^2(u^++u^-)+\frac14(u^++u^-)\big((u^+)^2+(u^-)^2\big),\label{eq:flux:fc}\\
		&\hat{g}(r^+, r^-)=\frac{1}{12}(r^++r^-)\big((r^+)^2+(r^-)^2\big).\label{eq:flux:gc}
		\end{align}
\end{enumerate}
	\end{subequations}
\end{itemize}

The definition of the algorithm is now complete. Based on the above choice of numerical fluxes, we introduce
	\begin{align}
	\Hcal_j^{\pm}(w,v)&=(w,v_x)_j - w_{\jfhalf}^{\pm}v_{\jfhalf}^{-} +  w_{\jbhalf}^{\pm}    v_{\jbhalf}^+,\\
	\Gcal_j^+(r^2u, v)&=(r^2u,v_x)_j - \{ur\}_{\jr}r^+_{\jr}v_{\jfhalf}^{-} + \{ur\}_{\jl}r^+_{\jl}  v_{\jbhalf}^+\\
	\Gcal_j^-(u^2r, v)&=(u^2r,v_x)_j - \{ur\}_{\jr}u^-_{\jr}v_{\jfhalf}^{-} + \{ur\}_{\jl}u^-_{\jl}  v_{\jbhalf}^+\\
	\Gcal_j(\mathfrak{f}, v)&=(\mathfrak{f},v_x)_j - \hat{\mathfrak{f}}_{\jr}v_{\jfhalf}^{-} + \hat{\mathfrak{f}}_{\jl}  v_{\jbhalf}^+,
\end{align}
where $\mathfrak{f}$ can be replaced by $f(u)$ or $g(r)$. Furthermore, we omit the subscript $j$ to denote the sum over $j$. After summing the variational formulations (\ref{eq:scheme:original}) over all cells, we get the LDG scheme in the global form:
\vspace{-0.3cm}
\begin{subequations}\label{LDG:global}
\begin{align}
(m_{t}, v)-\Gcal(f(u), v)+\Hcal^+(p,v)-\Gcal^+(r^2u,v)-\Hcal^+(s,v)&=0,\label{LDG:11:local}\\
(p, q)+\Gcal^-(u^2r,q)&=0,\label{LDG:12:local}	\\
(r,\psi)+\Hcal^-(u,\psi)&=0,\label{LDG:13:local}	\\
(m,\varphi)-(u, \varphi)-\Hcal^+(r,\varphi)&=0.\label{LDG:14:local} \\
(s, w)+\Gcal(g(r),w)&=0,\label{LDG:15:local}		
\end{align}
\end{subequations}

To conclude this section, we collect several analytical tools required in the subsequent
stability and error arguments. In particular, we summarize the inverse estimates for the
polynomial space and the structural properties of the LDG differential operators. These
auxiliary results will provide the technical foundation for the conservation analysis and the
later a priori error estimates.
\begin{lemma}\label{lemma:inverse}
(Inverse inequalities) There exists an inverse constant $\nu=\nu(k)$, such that for any $v\in V_h^k$
\begin{align} \label{inverse}
\norm{v_x}_{j}  \le \nu (\rho h)^{-1} \norm{v}_j, \quad
\|v\|_{L^\infty(I_j)}\le \sqrt{\nu (\rho h)^{-1}}\norm{v}_{j}.
\end{align}
\end{lemma}
We refer to \cite{Ciarlet1978} for these standard inverse inequalities. The following discrete Sobolev
inequality will be repeatedly used to control the nonlinear terms arising in the error analysis:
\begin{lemma}\label{lemma:infty} \cite[Lemma 4.1]{Wang2024}
There exists a constant $C$ independent of $h$, such that for any $v\in V_h^k$
\begin{align} \label{infty}
\|v\|^2_{\infty}\leq C\|v\|(\|v\|+\|v_x\|+h^{-\frac12}\normb{v}).
\end{align}
\end{lemma}
In the next lemmas, we recall some properties of bilinear forms $\Hcal^{\pm}$.
\begin{lemma}\label{negative:semidef}
For any $w, v \in H^1(\Omega_h)$, there holds
\begin{align}
&\Hcal^-(w,v)+\Hcal^+(v,w)=0, \label{skew}\\
&{|\Hcal^\pm(v,w)|\le  \left(\|v\|+\sqrt{\nu^{-1}(\rho h)}\|v\|_{\partial \Omega_h}\right)\left( \norm{w_x}+\sqrt{\nu (\rho h)^{-1}}\normb{w}\right)}. \label{bond:1}
\end{align}
\end{lemma}
\begin{proof}
The proof is the standard argument in the DG framework; thus, we omit it and refer to  \cite{Zhang2012} for more details.
\end{proof}
The following lemma establishes the coupling between the auxiliary variables and the
original unknown. This estimate is a crucial ingredient in the treatment of the projection
errors and nonlinear residual terms appearing in the convergence proof.
\begin{lemma}\label{lemma:ux:qx}
For $w\in V_h^k$ and $f\in L^2(\Omega)$, if $\Hcal_j^\pm(w,v)=(f,v)_j$ $\forall v\in V_h^k, \, j=1, \cdots, N$, then there exists a positive constant $\Cnurho$ dependent on $\nu$ and $\rho$, such that
\begin{align}
& \norm{w_x}+ \sqrt{\nu (\rho h)^{-1}}\normb{w} \le \Cnurho \norm{f}. \label{w_x}
\end{align}
\end{lemma}
\begin{proof}
We refer to \cite{Wang2015} for the details of the proof.
\end{proof}			
\section{The conservation and stability of the LDG scheme}\label{sec3}\setcounter{equation}{0}
In this section, we study the conservation and stability of the LDG scheme \eqref{eq:scheme:original} for solving the mCH equation \eqref{eq:mch}.
\begin{theorem}\label{thm:stab}
Let $u$ and $r$ be the solutions of the scheme \eqref{LDG:1:local}-\eqref{LDG:5:local}, and let $m$ be the numerical momentum density, 
then the numerical scheme is conservative with respect to $m$, that is
\begin{align}\label{eq:conservation}
    \frac{d}{dt}\int_{\Omega}m \,dx =0,
\end{align}
and the discrete energy $E(u, r)=\|u\|^2+\|r\|^2$ satisfies:
\begin{itemize}
  \item For the dissipative scheme with the numerical fluxes \eqref{eq:flux:alter} and \eqref{eq:flux:fd}{\rm :}
  \begin{align}
  \frac{d}{dt}E(u, r)\leq 0. \label{stability:ineq1}
  \end{align}
  \item
  For the conservative scheme with the numerical fluxes \eqref{eq:flux:alter} and \eqref{eq:flux:fc}{\rm :}
  \begin{align}
  \frac{d}{dt}E(u, r)=0. \label{stability:ineq2}
  \end{align}
\end{itemize}
\end{theorem}				
\begin{proof}
For the conservation, we take $v=1$ in \eqref{LDG:11:local} to immediately obtain \eqref{eq:conservation}.

\noindent \underline{\textbf{The first energy equation.}} We choose $v=u$, $q=-r$ and $\psi=p$ in \eqref{LDG:11:local}-\eqref{LDG:13:local}, respectively, to obtain
\begin{subequations}\label{stability:global1}
		\begin{align}
			(m_{t}, u)-\Gcal(f(u), u)+\Hcal^+(p,u)-\Gcal^+(r^2u,u)-\Hcal^+(s,u)&=\!0,\label{stability:11:local}			\\
			-(p, r)-\Gcal^-(u^2r, r)&=\!0,\label{stability:12:local}	\\
			(r, p)+\Hcal^-(u, p)&=\!0.\label{stability:13:local}
		\end{align}
	\end{subequations}
By summing up the above three equations in \eqref{stability:global1}, it follows from Lemma \ref{negative:semidef} that
	\begin{align*}
(m_{t}, u)-\Gcal(f(u), u)-\Gcal^+(r^2u,u)-\Gcal^-(u^2r, r)-\Hcal^+(s,u)&=0.
	\end{align*}
By the definition of $\Gcal^{\pm}$, it is not hard to get
\begin{align}
&\Gcal^+(r^2u,u)+\Gcal^-(u^2r, r)=0,\label{eq:nonstability}
\end{align}
Therefore, we have
\begin{align}
(m_{t}, u)-\Gcal(f(u), u)-\Hcal^+(s,u)=0. \label{eq1:stab}
\end{align}
We choose $\psi=-s$ and $w=r$ in \eqref{LDG:13:local} and \eqref{LDG:15:local} respectively, to obtain
\begin{subequations}\label{stability:global2}
		\begin{align}
			-(r, s)-\Hcal^-(u, s)&=\!0,\label{stability:21:local}	\\
			(s, r)+\Gcal(g(r), r)&=\!0.\label{stability:22:local}
		\end{align}
	\end{subequations}
By summing up \eqref{eq1:stab} and \eqref{stability:global2}, it follows from Lemma \ref{negative:semidef} that
\begin{align}
(m_{t}, u)-\Gcal(f(u), u)+\Gcal(g(r), r)=0. \label{eq2:stab}
\end{align}
By the definition of $\Gcal$, it is not hard to get
\begin{align*}
\Gcal(f(u), u)=-\sumn\Theta_{\jr},\quad \Gcal(g(r), r)=-\sumn \Phi_{\jr},
\end{align*}
where
\begin{align*}
&\Theta_{\jr}=\int_{u^-_{\jr}}^{u^+_{\jr}}f(s)-\hat{f}(u_{\jr}^-,u_{\jr}^+)\,ds,\\
&\Phi_{\jr}=\int_{r^-_{\jr}}^{r^+_{\jr}}g(s)-\hat{g}(r_{\jr}^-,r_{\jr}^+)\,ds
\end{align*}
Therefore, we have
\begin{align}
(m_{t}, u)=-\sumn\Theta_{\jr}+\sumn\Phi_{\jr}. \label{stability:global21}
\end{align}
In addition, it is easy to check $\displaystyle\sumn\Theta_{\jr}\geq0$ and $\displaystyle\sumn\Phi_{\jr}\leq 0$ for the flux \eqref{eq:flux:fd}  and $\displaystyle\sumn\Theta_{\jr}=\displaystyle\sumn\Phi_{\jr}=0$ for the flux \eqref{eq:flux:fc}.

\noindent \underline{\textbf{The second energy equation.}}
We first choose $\psi=r_t$ in \eqref{LDG:13:local}, then we take the time derivative in \eqref{LDG:14:local} and choose $\varphi=-u$ in \eqref{LDG:14:local}, to get
\begin{subequations}\label{stability:global3}
		\begin{align}
			(r, r_t)+\Hcal^-(u, r_t)&=0,\label{stability:14:local}\\
			-(m_t, u)+(u_t, u)+\Hcal^+(r_t, u)&=0.\label{stability:15:local}
		\end{align}
\end{subequations}
By summing up the above two equations in \eqref{stability:global3}, we obtain from Lemma \ref{negative:semidef}
\begin{align}
			(r, r_t)-(m_t, u)+(u_t, u)=0.\label{stability:global4}
\end{align}
Therefore, by \eqref{stability:global21} and \eqref{stability:global4}, we have
\begin{align}
			(r, r_t)+(u, u_t)=-\sumn\Theta_{\jr}+\sumn\Phi_{\jr}, \label{stability:global5}
\end{align}
which yields \eqref{stability:ineq1} and \eqref{stability:ineq2}.
\end{proof}					
\section{Error estimates of the LDG method}
\label{sec4}
\setcounter{equation}{0}

This section develops the a priori error analysis of the LDG approximation \eqref{eq:scheme:original} to the
mCH model \eqref{eq:mch}. The argument is presented first for the dissipative formulation, since the
essential stability mechanism and projection estimates are already contained in this case.
The corresponding estimates for the conservative formulation follow by the same framework
with only minor modifications. For the dissipative scheme, we have
$$\Gcal(f(w), v)=\Hcal^-(2\kappa^2 w, v)+\Hcal^-(w^3, v),\quad \Gcal(g(w), v)=\Hcal^+(g(w), v) \quad \forall v\in V_h^k.$$
We assume that the exact solution $U(x,t)$ satisfies the following regularity assumption
\begin{align}\label{smooth:assumption}
	{U, U_t \in L^{\infty}(0,T;H^{k+3}(\Omega)).}
\end{align}
\subsection{Projections}
To establish the convergence result, we introduce several local projection operators. These
projections separate the approximation error caused by the polynomial space from the
discrete error generated by the numerical scheme.

$\bullet$ The $L^2$ projection $P_h$.  For $\forall w\in L^2(\Omega)$, $P_hw\in V_h^k$ is defined as follows: In each interval $I_j$, there holds
\begin{align*}
	 (P_hw-w, v)_{j}=0\quad \forall v\in \mathcal{P}_{k}(I_j).
\end{align*}

$\bullet$ The Gauss-Radau projections $P_h^{\pm}$. For $\forall w\in \Ht\, (m\geq1)$, $P_h^{\pm}w\in V_h^k$ is defined as follows: In each interval $I_j$, there holds
\begin{equation*}
       (P_h^{\pm}w-w, v)_{j}=0 \quad \forall v\in \mathcal{P}_{k-1}(I_j),\quad
       (P_h^{\pm}w)_{j\mp\frac12}^{\pm}=w_{j\mp\frac12}^{\pm}.
\end{equation*}

By a standard scaling argument \cite{Ciarlet1978}, it is easy to obtain
the following approximation property for the projection errors
\begin{equation} \label{estimate:eta}
	h^{l}\|w-\pi_h w\|_{H^{l}(I_j)} +h^{\frac12} \norm{w-\pi_h w}_{L^\infty(I_j)} \leq
	C h^{\min(k+1,m)} \norm{w}_{H^m(I_j)}, \quad
\end{equation}
where $0\leq l\leq m$, $j=1,\cdots, N$, $\pi_h=P_h, P_h^\pm$ and $C>0$ is a bounded constant independent of $h$ and $j$.
Furthermore, from the definition of the projections, we can easily get
\begin{equation}\label{property:projection}
(w-P_hw, v)=0,\quad \Hcal^\pm(w-P_h^\pm w,v)=0 \quad \forall v\in V_h^k.
\end{equation}
\subsection{Error equations}
Let the total discretization errors be represented by
\begin{equation*}
	(e_u, e_p, e_r, e_m,e_s)=(U-u, P-p, R-r, M-m,S-s).
\end{equation*}
Using the projection operators defined above, each error variable is decomposed into an
interpolation component and a numerical component.
\begin{align*}
	(e_u, e_p, e_r, e_m,e_s) =(\eta_u-\xi_u, \eta_p-\xi_p, \eta_r-\xi_r, \eta_m-\xi_m,\eta_s-\xi_s),
\end{align*}
where
\begin{align*}
       \eta_u&=U-P_h^-U,~\,\qquad \xi_u=u-P_h^-U;\\
	\eta_p&=P-P_h^+P,~\,\qquad \xi_p=p-P_h^+P;\\
	\eta_r&=R-P_hR,~~\,\qquad \xi_r=r-P_hR;\\
	\eta_m&=M-P_hM,\qquad \xi_m=m-P_hM;\\
	\eta_s&=S-P_h^+S,~~\,\qquad \xi_s=s-P_h^+S.
\end{align*}
Note that the exact solutions $(U, P, R, M,S)$ also satisfy the LDG scheme \eqref{LDG:11:local}-\eqref{LDG:15:local}, hence we have the following error equations: For any test functions $v, q, \psi, \varphi,w \in V_h^k$,
\begin{subequations}\label{erreq}
\begin{align}
((\xi_m)_t, v)&=((\eta_m)_t, v)-\Hcal^-(2\kappa^2(\eta_u-\xi_u), v)-\Hcal^-(U^3-u^3, v)+\Hcal^+(\eta_p-\xi_p, v)\nonumber\\
&~\quad-\!\Gcal^+(R^2U\!-\!r^2u, v)\!-\Hcal^+(\eta_s-\xi_s,v),\label{erreq:1}			\\
(\xi_p, q)&=(\eta_p, q)+\Gcal^-(U^2R-u^2r,q),\label{erreq:2}	\\
(\xi_r,\psi)&=(\eta_r,\psi)+\Hcal^-(\eta_u-\xi_u,\psi),\label{erreq:3}	\\
(\xi_m,\varphi)-(\xi_u, \varphi)&=(\eta_m,\varphi)-(\eta_u,\varphi)-\Hcal^+(\eta_r-\xi_r,\varphi),\label{erreq:4}	\\
(\xi_s, w)&=(\eta_s, w)+\Hcal^+(g(R)-g(r),w). \label{erreq:5}
\end{align}
\end{subequations}
By using \eqref{property:projection}, we can simplify error equations \eqref{erreq:1}-\eqref{erreq:5} to obtain
\begin{subequations}\label{erreqv2}
\begin{align}
((\xi_m)_t, v)&=((\eta_m)_t, v)+\Hcal^-(2\kappa^2\xi_u, v)-\Hcal^-(U^3-u^3, v)+\Hcal^-(\xi_p, v)\nonumber\\
&~\quad-\!\Gcal^+(R^2U\!-\!r^2u, v)\!+\Hcal^+(\xi_s,v),\label{erreqv2:1}			\\
(\xi_p, q)&=(\eta_p, q)+\Gcal^-(U^2R-u^2r,q),\label{erreqv2:2}	\\
(\xi_r,\psi)&=-\Hcal^-(\xi_u,\psi),\label{erreqv2:3}	\\
(\xi_m,\varphi)-(\xi_u, \varphi)&=(\eta_m,\varphi)-(\eta_u,\varphi)-\Hcal^+(\eta_r-\xi_r,\varphi),\label{erreqv2:4}	\\
(\xi_s, w)&=(\eta_s, w)+\Hcal^+(g(R)-g(r),w). \label{erreqv2:5}
\end{align}
\end{subequations}
	
By Lemma \ref{lemma:infty}, Lemma \ref{lemma:ux:qx} and  (\ref{property:projection}), we can get the following corollary, which states the important relationships between $\xi_u$ and $\xi_r$.
\begin{crl}\label{corollary:xi}
	Suppose $\xi_u$ and $\xi_r$ satisfy (\ref{erreq:3}), then  we have
\begin{align}
	&\norm{(\xi_u)_x}+ \sqrt{\nu (\rho h)^{-1}}\normb{\xi_u} \le
	{\Cnurho} \norm{\xi_r}, \label{xi_u_x}\\
	&\|(\xi_u)\|_{\infty} \le C(\|\xi_u\|+\norm{\xi_r}), \label{xi_u_x2}
\end{align}
where $C$ is a constant independent of $h$.
\end{crl}

{Before presenting the energy estimates, let us first discuss the setting of the numerical initial condition.}
\subsection{The numerical initial condition}
The initial condition plays an important role in the proof of the error estimates. First, we take
\begin{align}
u(0)=P_h^-U_0.\label{eq:initial:c1}
\end{align}
It is noted that $r(0)$ can be obtained by the scheme \eqref{LDG:3:local}. By taking $u=P_h^-U_0$ in \eqref{LDG:3:local}, thanks to the definition of the projection $P_h^-$, we obtain
\begin{align}
r(0)=P_h(R(x,0)),\label{eq:initial:c2}
\end{align}
where $R(x,0)= U_0^{\prime}(x)$. Thus, we can easily get the following initial error estimates.
\begin{lemma}\label{lemma:initial}
	Assume that the initial condition $U_0(x) \in H^{1}(\Omega)$, and the numerical initial conditions  $u(0), r(0)$ satisfy \eqref{eq:initial:c1} and \eqref{eq:initial:c2}, respectively, then
	\begin{equation}
		\norm{\xi_u}(0)=0,\quad \norm{\xi_r}(0)=0.
	\end{equation}
\end{lemma}
\subsection{Error analysis}
We first give the following lemma, which presents the energy equation for $\xi_u$ and $\xi_r$.
\begin{lemma} \label{lemma:energy:eq}
The following equation holds:
	\begin{equation}\label{err:energy:eq}
	\begin{split}
		\frac12\frac{d}{dt}E(\xi_u, \xi_r)-\Hcal^-(2\kappa^2\xi_u, \xi_u)\!=\!&-\!(\eta_p, \xi_r)\!+\!((\eta_u)_t,\xi_u)\!+(\eta_s,\xi_r)+\!\Hcal^+((\eta_r)_t, \xi_u)\!\\
		&~\quad-\!\Hcal^-(U^3\!-\!\!u^3, \xi_u)+\Hcal^+(g(R)-g(r),\xi_r)\\
	&~\quad-\Gcal^-(U^2R-u^2r, \xi_r)-\!\Gcal^+(R^2U\!-\!r^2u, \xi_u).
	\end{split}
	\end{equation}
\end{lemma}
\begin{proof}
Taking $v=\xi_u$, $q=-\xi_r$, and $\psi=\xi_p$ in \eqref{erreq:1}-\eqref{erreq:3}, respectively, and owing to (\ref{property:projection}), we have
\begin{subequations}\label{err:energy:eq1}
\begin{align}
((\xi_m)_t, \xi_u)&=\Hcal^-(2\kappa^2\xi_u, \xi_u)-\Hcal^-(U^3\!-\!u^3, \xi_u)-\Hcal^+(\xi_p, \xi_u)-\!\Gcal^+(R^2U\!-\!r^2u, \xi_u)\\
&~\quad+\Hcal^+(\xi_s, \xi_u),\nonumber\\
-(\xi_p, \xi_r)&=-(\eta_p, \xi_r)-\Gcal^-(U^2R-u^2r, \xi_r),	\\
(\xi_r, \xi_p)&=-\Hcal^-(\xi_u, \xi_p).
\end{align}
\end{subequations}
By summing up the above three equations in \eqref{err:energy:eq1}, and using Lemma \ref{negative:semidef},  we get
\begin{equation}
	\begin{split}\label{err:energy:eq2}
((\xi_m)_t, \xi_u)&=\Hcal^-(2\kappa^2\xi_u, \xi_u)-\Hcal^-(U^3\!-\!u^3, \xi_u)\!+\Hcal^+(\xi_s, \xi_u)\\
	&~\quad-(\eta_p, \xi_r)-\Gcal^-(U^2R-u^2r, \xi_r)-\!\Gcal^+(R^2U\!-\!r^2u, \xi_u).
\end{split}
\end{equation}
Next, we choose $\psi=-\xi_s$ and $w=\xi_r$ in \eqref{erreq:3} and \eqref{erreq:5} respectively, to obtain
\begin{subequations}\label{err:energy:eq21}
\begin{align}
-(\xi_r, \xi_s)&=\Hcal^-(\xi_u, \xi_s),\\
(\xi_s, \xi_r)&=(\eta_s, \xi_r)+\Hcal^+(g(R)-g(r),\xi_r).
\end{align}
\end{subequations}
By summing up the above two equations in \eqref{err:energy:eq21} and \eqref{err:energy:eq2}, and employing Lemma \ref{negative:semidef}, we have
\begin{equation}
    \begin{split}
((\xi_m)_t, \xi_u)&=\Hcal^-(2\kappa^2\xi_u, \xi_u)-\Hcal^-(U^3\!-\!u^3, \xi_u)\!+(\eta_s,\xi_r)+\Hcal^+(g(R)-g(r),\xi_r)\\
	&~\quad-(\eta_p, \xi_r)-\Gcal^-(U^2R-u^2r, \xi_r)-\!\Gcal^+(R^2U\!-\!r^2u, \xi_u).
\end{split}
\end{equation}
Next, we choose $\psi=(\xi_r)_t$ in \eqref{erreq:3} and take the time derivative in \eqref{erreq:4} and choose $\varphi=-\xi_u$. Owing to (\ref{property:projection}) we have
\begin{equation}\label{err:energy:eq3}
\begin{split}
 (\xi_r, (\xi_r)_t)&=-\Hcal^-(\xi_u, (\xi_r)_t),\\
 -((\xi_m)_t,\xi_u)+((\xi_u)_t, \xi_u)&=((\eta_u)_t, \xi_u)+\Hcal^+((\eta_r)_t, \xi_u)-\Hcal^+((\xi_r)_t, \xi_u).
\end{split}
\end{equation}
By summing up the above two equations in \eqref{err:energy:eq3}, and employing Lemma \ref{negative:semidef},  we get
\begin{equation}
	\begin{split}
\label{err:energy:eq4}
 (\xi_r, (\xi_r)_t)-((\xi_m)_t,\xi_u)+((\xi_u)_t, \xi_u)=((\eta_u)_t, \xi_u)+\Hcal^+((\eta_r)_t, \xi_u).
\end{split}
\end{equation}
Combining \eqref{err:energy:eq2} and \eqref{err:energy:eq4}, we obtain \eqref{err:energy:eq}.
\end{proof}
It remains to control each contribution appearing on the right-hand side of the energy
identity. We introduce the notation below to treat the linear projection terms and nonlinear
consistency terms separately. Next, we need to estimate the terms on the right-hand side of \eqref{err:energy:eq} to obtain the estimate for $\xi_u$ and $\xi_r$. Thus, we denote
\begin{align*}
\Theta_1&:=-(\eta_p, \xi_r)+((\eta_u)_t,\xi_u)+(\eta_s,\xi_r)+\Hcal^+((\eta_r)_t,\xi_u);\\
\Theta_2&:=-\Hcal^-(U^3-u^3, \xi_u);\\
\Theta_3&:=\Hcal^+(g(R)-g(r),\xi_r);\\
\Theta_4&:=-\Gcal^+(R^2U\!-\!r^2u, \xi_u)-\Gcal^-(U^2R-u^2r, \xi_r).
\end{align*}
In the estimates of the $\Theta_1$ - $\Theta_4$ we assume $k\geq 1$ and $h<1$. In addition, under the smoothness assumption \eqref{smooth:assumption}, the constant ``$C$" in Lemma \ref{lemma:estimate1} - Lemma \ref{lemma:estimate6} depends on the smoothness of the exact solution and is independent of $h$.

\begin{lemma} {\bf{(The estimate for $\Theta_1$)}}\label{lemma:estimate1}
For $k\geq1$, we have the following estimates for the term $\Theta_1$
{
\begin{align}
\Theta_1&\leq C\|\xi_u\|^2+C\|\xi_r\|^2+Ch^{2k+2},\label{eq:est:theta1}
\end{align}
}
where $C$ is a constant independent of $h$.
\end{lemma}
\begin{proof}
According to the projection properties \eqref{estimate:eta}, Lemma \ref{negative:semidef} and Corollary \ref{corollary:xi}, we get the estimates for $\Theta_1$.
\end{proof}
\begin{lemma} {\bf{(The estimate for $\Theta_2$)}}\label{lemma:estimate2}
For $k\geq1$, we have the following estimates for the term $\Theta_2$
{
\begin{align}
\Theta_2&\leq C\|\xi_u\|^2+C\|\xi_r\|^2+Ch^{2k+2},\label{eq:est:theta2}
\end{align}
}
where $C$ is a constant independent of $h$.
\end{lemma}
\begin{proof}
{To estimate $\Theta_2$, we first rewrite the error $U^3-u^3$ in the following form:
\begin{align*}
U^3-u^3=e_u(3U^2-3Ue_u+e_u^2),
\end{align*}
from which we get
\begin{align}\label{eq:4.4.1}
\Theta_2=-\Hcal^-(3U^2e_u, \xi_u)+\Hcal^-(3Ue_u^2, \xi_u)-\Hcal^-(e_u^3, \xi_u).
\end{align}
For the first term in \eqref{eq:4.4.1}, by Lemma \ref{negative:semidef}, Corollary \ref{corollary:xi} and inverse inequality for $\xi_u$, we have
\begin{align*}
-\Hcal^-(3U^2e_u, \xi_u)&\leq C (\|\eta_u\|+h^{1/2}\|\eta_u\|_{\partial \Omega_h}+\|\xi_u\|)\|\xi_r\|\\
&\leq C\|\xi_u\|^2+C\|\xi_r\|^2 +Ch^{2k+2}.
\end{align*}
The last inequality is derived by the error estimate of projections \eqref{estimate:eta}.  For the second term in  \eqref{eq:4.4.1}, by Lemma \ref{lemma:infty}, Lemma \ref{lemma:ux:qx} and \eqref{LDG:13:local}, we have
\begin{align*}
\|u\|_{\infty}^2&\leq C\|u\|(\|u\|+\|u_x\|+h^{-1/2}\normb u )\leq C\|u\|(\|u\|+\|r\|).
\end{align*}
From the energy stability result in Theorem \ref{thm:stab} and the boundedness of projections \eqref{eq:initial:c1}-\eqref{eq:initial:c2}, we obtain
\begin{align*}
\|u\|_{\infty}^2& \leq C (\|u(0)\|+\|r(0)\|)^2\leq C(\|U_0\|_{\infty}+\|U_0'\|)^2.
\end{align*}
Thus we have an estimate for $\|e_u\|_{\infty}$:
\begin{align}\label{est:eu}
\|e_u\|_{\infty}\leq \|U\|_{\infty}+\|u\|_{\infty}\leq \|U\|_{\infty}+C (\|U_0\|_{\infty}+\|U_0'\|)\leq C.
\end{align}
Furthermore, by Lemma \ref{negative:semidef} and Corollary \ref{corollary:xi}, we have
\begin{align*}
\Hcal^-(3Ue_u^2, \xi_u)&\leq C(\|e_u^2\|+h^{1/2}\|e_u^2\|_{\partial \Omega_h})\|\xi_r\|\\
&\leq C(\|e_u\|+h^{1/2}\|e_u\|_{\partial \Omega_h})\|\xi_r\|\quad (\text{by \eqref{est:eu}})\\
&\leq C\|\xi_u\|^2+C\|\xi_r\|^2 +Ch^{2k+2}.
\end{align*}
For the third term in \eqref{eq:4.4.1}, we can do a similar analysis as that for the second term. This completes the proof.
}
\end{proof}
To estimate $\Theta_3$, we need to use an a priori assumption, which is usually used in error estimates for nonlinear equations. We assume that 
\begin{align}\label{pri_ass}
    \|R-r\|\leq h^{\frac 12}.
\end{align}
\begin{lemma} {\bf{(The estimate for $\Theta_3$)}}\label{lemma:estimate3}
For $k\geq1$, we have the following estimates for the term $\Theta_3$
{
\begin{align}
\Theta_3&\leq C\|\xi_r\|^2+Ch^{2k},\label{eq:est:theta3}
\end{align}
}
where $C$ is a constant independent of $h$.
\end{lemma}
\begin{proof}
By Taylor expansion, we have
\begin{align*}
    g(R)-g(r)=g'(R)(R-r)-\frac{1}{2}g''(r+\theta(R-r))(R-r)^2 \text{ for some } \theta \in (0,1).
\end{align*}
By \eqref{bond:1}, \eqref{estimate:eta}, and \eqref{pri_ass} we have
\begin{align*}
    \Theta_3\leq C(\|\xi_r\|^2+h^{2k}).
\end{align*}
\end{proof}
\begin{remark}
We note that, unlike in \cite{Xu2008CH}, estimating $\Theta_2$ does not require any a priori error assumptions. This is attributable to the polynomial nature of the nonlinearity and, more importantly, the intrinsic relation between $u$ and $r$. Energy stability furnishes the requisite $L^\infty$-boundedness of $u$, enabling us to control the higher-order terms present in $\Theta_2$. Moreover, by invoking the relation between $\xi_u$ and $\xi_r$ once more, we readily handle both the derivative of $\xi_u$ and the associated boundary terms, which renders the estimates significantly more tractable.

In contrast, for $\Theta_3$, we only possess boundedness of $\|r\|_{L^2}$ rather than $\|r\|_{L^\infty}$. Consequently, an a priori assumption on $r$ is still required to manage the nonlinear terms $g(r)$---a strategy that is standard in error estimates for nonlinear equations. This assumption can be rigorously justified for $k >\frac{1}{2}$ via a continuity argument, following the approach in \cite{Xu2008CH}.
\end{remark}
The estimate for the $\Theta_4$ are very technical since they include nonlinear differential terms and nonlinear boundary terms. The main idea in our analysis is to make use of the nonlinear stability as given in \eqref{eq:nonstability}. However, since the stability results are only valid for functions in $V_h^k$, we need to decompose the error with the help of projections and use the following property
\begin{align}\label{eq:nonlinear:xi}
\Gcal^-(\xi_r^2\xi_u, \xi_u)+\Gcal^+(\xi_u^2\xi_r, \xi_r)&=0.
\end{align}
We use the following lemma to estimate $\Theta_4$.
\begin{lemma}\label{lemma:estimate4} {\bf{(The estimate for $\Theta_4$)}}
There exists a constant $C$ independent of $h$, such that for $k\geq1$
\begin{equation}\label{eq:est:theta4}
\begin{split}
\Theta_4\leq &~ Ch^{-1}(\|\xi_r\|^4+\|\xi_u\|^4)+C(\|\xi_r\|^2+\|\xi_u\|^2)+Ch^{k}\|\xi_u\|+Ch^{k}\|\xi_r\|,
\end{split}
\end{equation}
where $C$ is a constant independent of $h$.
\end{lemma}
\begin{proof}
We recall the definition of $\Theta_4$,
$$\Theta_4=-\Gcal^+(R^2U-r^2u, \xi_u)-\Gcal^-(U^2R-u^2r, \xi_r).$$
We perform an error decomposition to extract  $\Gcal^+(\xi_r^2\xi_u, \xi_u)$ and $\Gcal^-(\xi_u^2\xi_r, \xi_r)$, and the summation of these two terms will vanish due to \eqref{eq:nonlinear:xi}.

\noindent {\textbf{Step 1: Error decomposition.}}\\
Firstly, we have
\begin{align*}
R^2-r^2=e_r(2R-e_r)=\eta_r(2R-\eta_r)+2\eta_r\xi_r-2R\xi_r-\xi_r^2.
\end{align*}
The projection errors $\eta_u$ and $\eta_r$ are high-order terms, since we have the projection error estimates \eqref{estimate:eta}.
Therefore, we put together the terms containing projection errors and denote $$A_1=\eta_r(2R-\eta_r)+2\eta_r\xi_r.$$
Therefore,
\begin{align*}
&-(R^2U-r^2u)\\
=&-U(R^2-r^2)-R^2(U-u)+(U-u)(R^2-r^2)\\
=&-U(A_1-2R\xi_r-\xi_r^2)-R^2\eta_u+R^2\xi_u+(\eta_u-\xi_u)(A_1-2R\xi_r-\xi_r^2)\\
:=&\,\Pi_1+\Pi_2+\xi_u\xi_r^2,
\end{align*}
where
\begin{align*}
\Pi_1&=-UA_1-R^2\eta_u+\eta_u(A_1-2R\xi_r-\xi_r^2)-\xi_uA_1,\\
\Pi_2&=2UR\xi_r+U\xi_r^2+R^2\xi_u+2R\xi_r\xi_u.
\end{align*}
Similarly, we denote $A_2=\eta_u(2U-\eta_u)+2\eta_u\xi_u$, then
\begin{align*}
-(U^2R-u^2r):=\Pi_3+\Pi_4+\xi_r\xi_u^2,
\end{align*}
where
\begin{align*}
\Pi_3&=-RA_2-U^2\eta_r+\eta_r(A_2-2U\xi_u-\xi_u^2)-\xi_rA_2,\\
\Pi_4&=2UR\xi_u+R\xi_u^2+U^2\xi_r+2U\xi_u\xi_r.
\end{align*}
Hence, we have
\begin{align*}
\Theta_4&=\Gcal^+(\Pi_1+\Pi_2, \xi_u)+\Gcal^-(\Pi_3+\Pi_4, \xi_r)+\Gcal^+(\xi_r^2\xi_u, \xi_u)+\Gcal^-(\xi_u^2\xi_r, \xi_r)\\
&=\Gcal^+(\Pi_1+\Pi_2, \xi_u)+\Gcal^-(\Pi_3+\Pi_4, \xi_r)
\end{align*}
After the error decomposition, we extract $\Gcal^+(\xi_r^2\xi_u, \xi_u)$ and $\Gcal^-(\xi_u^2\xi_r, \xi_r)$, and ensure that each term in $\Pi_1$ and $\Pi_3$ includes a projection error. Thus, it is easy to obtain the estimates for $\Pi_1$ and $\Pi_3$ by the projection properties and inverse inequalities.  However, the terms in  $\Pi_2$ and $\Pi_4$ should be treated carefully.

\noindent {\textbf{Step 2: Estimates.}} \\
The estimates for $\Gcal^+(\Pi_1, \xi_u)+\Gcal^-(\Pi_3, \xi_r)$:
\vspace{0.2cm}

By the projection property \eqref{estimate:eta} and inverse inequalities \eqref{inverse}, we have
\begin{align*}
\|\Pi_1\|&\leq Ch^{k+\frac12}\|\xi_r\|+Ch^{k+\frac12}\|\xi_u\|+Ch^{k}\|\xi_r\|\|\xi_u\|+Ch^{k}\|\xi_r\|^2+Ch^{k+1},\\
\|\Pi_3\|&\leq Ch^{k+\frac12}\|\xi_r\|+Ch^{k+\frac12}\|\xi_u\|+Ch^{k}\|\xi_r\|\|\xi_u\|+Ch^{k}\|\xi_u\|^2+Ch^{k+1}.
\end{align*}
Therefore,
\begin{align*}
&\Gcal^+(\Pi_1, \xi_u)+\Gcal^-(\Pi_3, \xi_r)\\
\leq&\, Ch^{-1}\|\Pi_1\|\|\xi_u\|+Ch^{-1}\|\Pi_3\|\|\xi_r\|\\
\leq&\, C\|\xi_r\|^2+C\|\xi_u\|^2+Ch^{k}\|\xi_u\|+Ch^{k}\|\xi_r\|+C\|\xi_r\|^4+C\|\xi_u\|^4.
\end{align*}
The estimates for $\Gcal^+(\Pi_2, \xi_u)+\Gcal^-(\Pi_4, \xi_r)$:
\vspace{0.2cm}

By integration by parts, we have
\begin{align*}
&\Gcal^+(2UR\xi_r, \xi_u)+\Gcal^-(2UR\xi_u, \xi_r)=-2((UR)_x\xi_u, \xi_r),\\
&\Gcal^+(U\xi_r^2+2R\xi_r\xi_u, \xi_u)+\Gcal^-(R\xi_u^2+2U\xi_r\xi_u, \xi_r)
=-(U_x\xi_r^2, \xi_u)-(R_x\xi_u^2, \xi_r)+\Gamma_1,
\end{align*}
where
\begin{align*}
\Gamma_1=&\sum_{j=1}^{N}\Big(\big(U\{\xi_r\}\xi_r^++2R\{\xi_u\}\xi_r^+\big)\jump{\xi_u}+\big(R\{\xi_u\}\xi_u^-+2U\{\xi_r\}\xi_u^-\big)\jump{\xi_r}\Big)_{\jr}\\
&+\sum_{j=1}^{N}\Big(U(\xi_r^-)^2\xi_u^--U(\xi_r^+)^2\xi_u^++R(\xi_u^-)^2\xi_r^--R(\xi_u^+)^2\xi_r^+\Big)_{\jr}\\
=&\sum_{j=1}^{N}\Big(\frac{1}{2}(U\xi_r^++R\xi_u^-)\jump{\xi_u}\jump{\xi_r}\Big)_{\jr}.
\end{align*}
{By the Cauchy-Schwarz inequality, we have
\begin{align*}
&\Gcal^+(2UR\xi_r+U\xi_r^2+2R\xi_r\xi_u, \xi_u)+\Gcal^-(2UR\xi_u+R\xi_u^2+2U\xi_r\xi_u, \xi_r)\\
=&~ -2((UR)_x\xi_u, \xi_r)-(U_x\xi_r^2, \xi_u)-(R_x\xi_u^2, \xi_r)+\Gamma_1\\
\leq&~C\|\xi_u\|\| \xi_r\|+\|\xi_r\xi_u\|(\|\xi_u\|+\|\xi_r\|)+\Gamma_1\\
\leq&~Ch^{-1/2}(\|\xi_r\|^2\|\xi_u\|+\|\xi_u\|^2\|\xi_r\|)+C(\|\xi_r\|^2+\|\xi_u\|^2)+\Gamma_1\\
\leq&~Ch^{-1}(\|\xi_r\|^4+\|\xi_u\|^4)+C(\|\xi_r\|^2+\|\xi_u\|^2)+\Gamma_1.
\end{align*}
Here we used the inverse inequality for $\|\xi_u\xi_r\|$ as follows:
\begin{align}\label{inv_est:A.1}
\|\xi_u\xi_r\|\leq \|\xi_u\|_{\infty}\|\xi_r\|\leq Ch^{-1/2}\|\xi_u\|\|\xi_r\|.
\end{align}}
For $\Gamma_1$, by Corollary \ref{corollary:xi}, the Cauchy-Schwarz inequality, and the inverse inequality, we have
\begin{align*}
\Gamma_1 &\leq C(\|\xi_r\|_{\infty}+\|\xi_u\|_{\infty})\|\xi_r\|h^{-1/2}\normb{\xi_u}\\
&\leq Ch^{-1/2}(\|\xi_r\|+\|\xi_u\|)\|\xi_r\|^2\\
&\leq Ch^{-1}\|\xi_r\|^{4}+C(\|\xi_r\|^2+\|\xi_u\|^2).
\end{align*}
Finally, the integration by parts gives
\begin{align*}
&\Gcal^+(R^2\xi_u, \xi_u)=-(RR_x,\xi_u^2)\leq C\|\xi_u\|^2,\\
&\Gcal^-(U^2\xi_r, \xi_r)=-(UU_x,\xi_r^2)\leq C\|\xi_r\|^2.
\end{align*}
Therefore,
\begin{align*}
&\Gcal^+(\Pi_2, \xi_u)+\Gcal^-(\Pi_4, \xi_r)\\
\leq&\, Ch^{-1}(\|\xi_r\|^4+\|\xi_u\|^4)+C(\|\xi_r\|^2+\|\xi_u\|^2).
\end{align*}
This completes the proof.
\end{proof}

\begin{lemma}\label{lemma:estimate6} 
For $k\geq1$, $\xi_u$ and $\xi_r$ satisfy
\begin{align}\label{eq:estimate6}
\frac{d}{dt}E(\xi_u, \xi_r)\leq  &Ch^{-1}(\|\xi_r\|^4+\|\xi_u\|^4)+C(\|\xi_r\|^2+\|\xi_u\|^2)+Ch^{2k}.
\end{align}
where $C$ is a positive constant independent of $h$.
\end{lemma}
\begin{proof}
Since $\Hcal^-(2\kappa^2\xi_u, \xi_u)\leq 0$ and by Lemma \ref{lemma:energy:eq} and combining the estimates for $\Theta_1$ -- $\Theta_4$ in Lemma \ref{lemma:estimate1} -- Lemma \ref{lemma:estimate4}, we have
\begin{align*}
\frac{d}{dt}E(\xi_u, \xi_r)\leq  &Ch^{-1}(\|\xi_r\|^4+\|\xi_u\|^4)+C(\|\xi_r\|^2+\|\xi_u\|^2)+Ch^{2k}.
\end{align*}
\end{proof}

\begin{lemma}\label{lemma:est:A}
If $k\geq 1$, $A(0)\leq Ch^{2k+2}$ and $A(t)$ satisfies the following inequality
  \begin{align}\label{ine:ch:3}
   A^\prime(t)\le C(h^{-1}A^2+A+h^{2k}), \quad 0\leq t \leq T,
   \end{align}
where $C$ is a constant independent of $h$ and $t$. Then when $h$ is small enough, we have
$$A(t)\le \tilde{C}h^{2k}, \quad 0\leq t \leq T,$$
where $\tilde{C}$ is a constant independent of $h$ and dependent on $T$.
\end{lemma}
\begin{proof}
The proof of this lemma can be found in \cite{Tao2025MC}.
\end{proof}
Finally, we present our main result in this section by the following Theorem.

\begin{theorem}
Let $(U, P, R, M,S)$ be the exact solution of the modified Camassa-Holm equation \eqref{eq:mch} satisfying the smoothness assumption \eqref{smooth:assumption},
and let $(u, p, r, m,s)$ be the numerical solution of the LDG scheme \eqref{eq:scheme:original}, then under the initial condition in Lemma \ref{lemma:initial} and for $k\geq1$, we have
	\begin{equation}\label{eq:est:mian}
		\|U-u\|^2+\| R-r\|^2 \le C h^{2k},
	\end{equation}
where $C$ is a positive constant independent of $h$ and dependent on the $\|U\|_{L^{\infty}([0,T], H^{k+3}(\Omega))}$.
\end{theorem}
\begin{proof}
By using Lemma \ref{lemma:estimate6}, Lemma \ref{lemma:est:A}, and the estimates for the initial condition in Lemma \ref{lemma:initial}, we have
\begin{align*}
		E(\xi_u, \xi_r) \le C h^{2k}.
\end{align*}
Combining the approximation property for the projection error and using the triangle inequality we get
\begin{align*}
		\|U-u\| +\| R-r\| &\leq \|\eta_u\|+\|\eta_r\|+\|\xi_u\|+\|\xi_r\|\leq Ch^{k}.
\end{align*}
This completes the proof.
\end{proof}

\section{Numerical experiments}
\label{sec5}
\setcounter{equation}{0}
In this section, a collection of numerical tests is performed to examine the practical performance of the proposed LDG method and to compare the computational observations
with the theoretical estimates. Unless explicitly stated otherwise, the temporal discretization is carried out by the classical fourth-order Runge–Kutta method under the time-step
restriction $\Delta t=O(h)$. The numerical accuracy is evaluated through the discrete energy
error $\sqrt{E(U-u,R-r)}$. The reported computations were partially carried out using the
high-performance computing facilities of the State Key Laboratory of Scientific and Engineering Computing, Chinese Academy of Sciences.
\begin{exmp}\label{example2}
In this example, we consider a smooth soliton solution of \eqref{eq:mch} with $\kappa=0$ as shown in \cite{MatsunoJMP2013}. The exact solution can be written in the following form
\begin{align*}
U=u_0+\frac{\alpha^2\tilde{c}(a\cosh(\xi-\xi_0)+1)}{u_0^2(\cosh(\xi-\xi_0)+a)^2},
\end{align*}
where
\begin{align*}
x-ct-x_0=\frac{\xi}{\alpha}+2\ln\left(\frac{\sqrt{1-\alpha}e^{\xi-\xi_0}+\sqrt{1+\alpha}}{\sqrt{1+\alpha}e^{\xi-\xi_0}+\sqrt{1-\alpha}}\right),
\end{align*}
with
\begin{align*}
a=\frac{1}{\sqrt{1-\alpha^2}},\quad \tilde{c}=\frac{2u_0^3}{1-\alpha^2},\quad \xi_0=\frac 12 \ln\left(\frac{1+\alpha}{1-\alpha}\right), \quad c=\frac{\tilde{c}}{u_0}+u_0^2.
\end{align*}
\end{exmp}

The first experiment is designed to assess the convergence behavior for smooth solitary
waves. We choose $u_0=1$, $x_0=0$, and investigate three representative amplitudes,
 $\alpha=0.5,0.7,0.8$, on the uniform mesh over $\Omega=(-20,20)$ until $T=0.5$. The prescribed
boundary data remain equal to the background state $u_0$. Figure \ref{fig_ex2_1} (left) displays the
corresponding profiles. As $\alpha$ approaches the critical value $\sqrt{3}/2$, the wave crest becomes
increasingly steep and the solution approaches a nonsmooth limiting state. The convergence
tables demonstrate that the observed rates agree with the predicted $k$-th order accuracy
for the regular cases. The mild oscillations of the rates for $\alpha =0.8$ are attributed to the
proximity of the solution to the singular regime.
\begin{figure}
  \centering
  \includegraphics[width=0.45\textwidth]{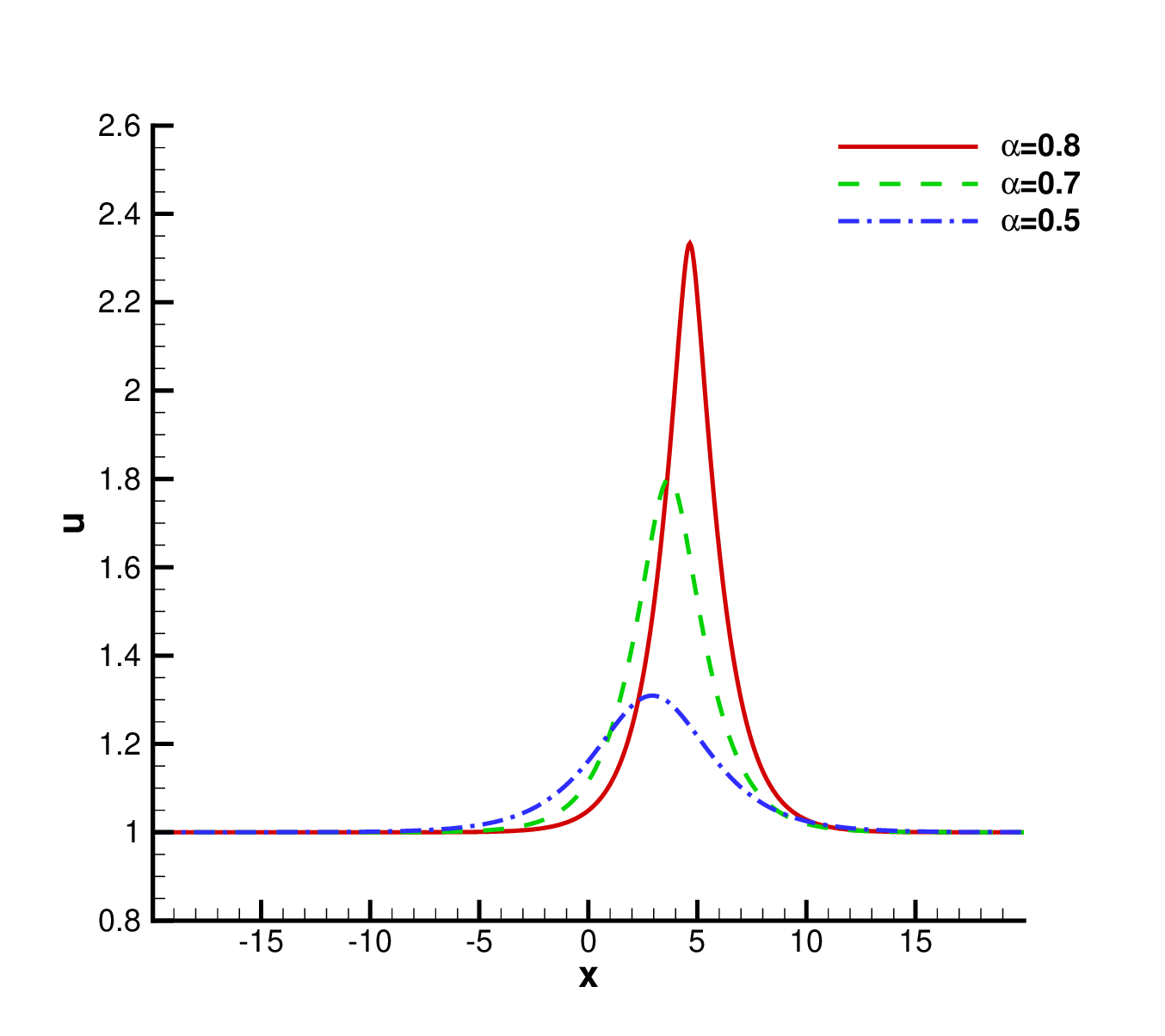}
  \includegraphics[width=0.45\textwidth]{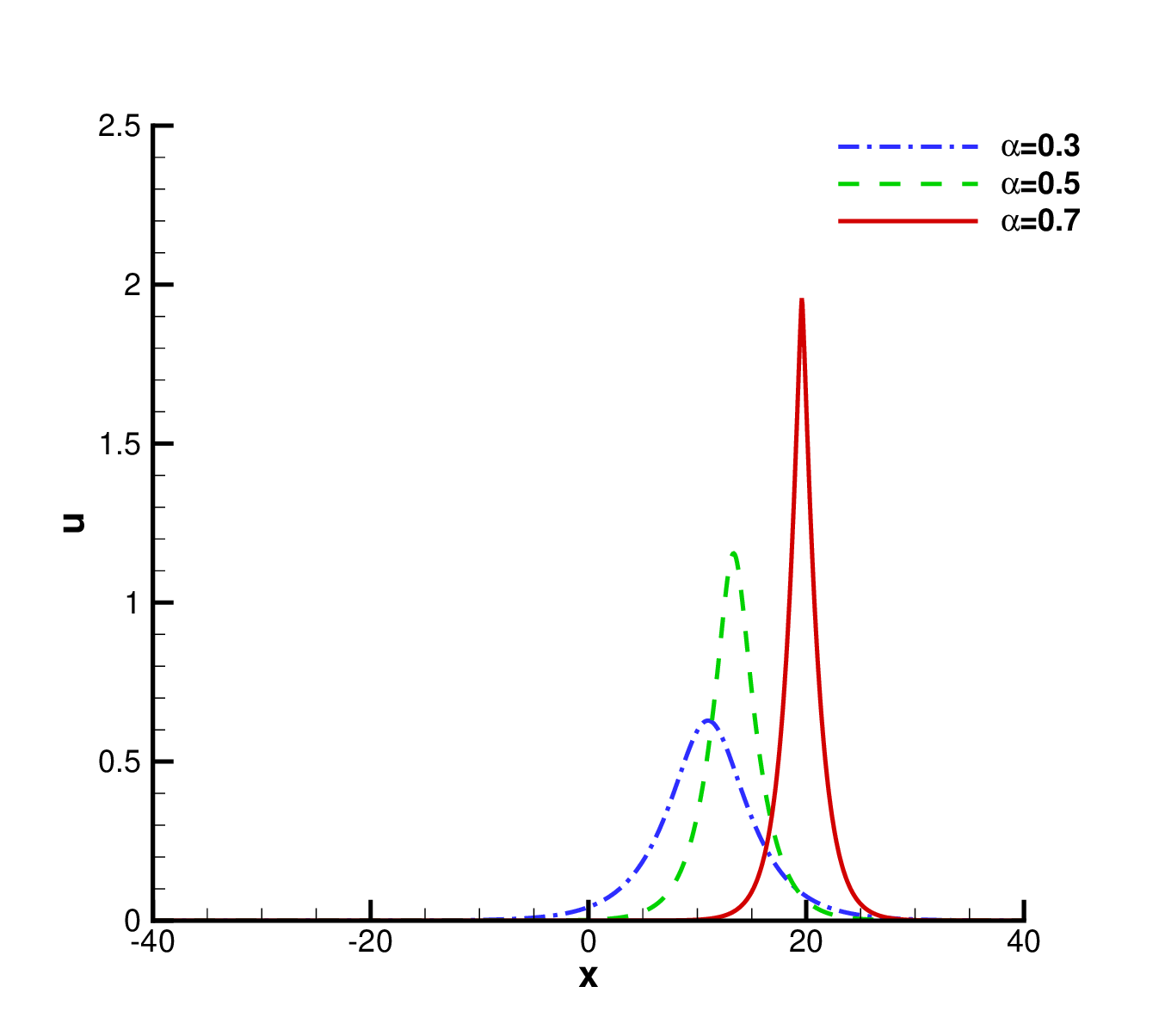}
  \caption{{ Example \ref{example2}: The exact solution at time $T=0.5$ for $\alpha=0.5,0.7,0.8$ (left). Example \ref{example_kappa_1}: The exact solution at time $T=5.0$ for $\alpha=0.3,0.5,0.7$ (right).}}\label{fig_ex2_1}
\end{figure}

\begin{table}[!ht]\centering

\begin{tabular}{|c|c|cc|cc|cc|}
  \hline
  & & \multicolumn{2}{c}{$\alpha=0.5$}&\multicolumn{2}{|c}{$\alpha=0.7$}&\multicolumn{2}{|c|}{$\alpha=0.8$}\\ \hline
  & $N$ & Error & Order & Error  & Order& Error &Order\\ \hline
  \multirow{6}{0.6cm}{$\mathcal{P}_0$}& $   20$ &    1.54E-01&     --    &    7.61E-01&     --     &    1.67E+00&     --    \\ \cline{2-8}
& $   40$ &    7.55E-02&     1.03  &    4.47E-01&     0.77	 &    1.22E+00&     0.45  \\ \cline{2-8}
& $   80$ &    3.62E-02&     1.06  &    2.19E-01&     1.03	 &    7.23E-01&     0.76  \\ \cline{2-8}
& $  160$ &    1.76E-02&     1.04  &    9.90E-02&     1.14	 &    3.57E-01&     1.02  \\ \cline{2-8}
& $  320$ &    8.68E-03&     1.02  &    4.53E-02&     1.13	 &    1.58E-01&     1.18  \\ \cline{2-8}
& $  640$ &    4.31E-03&     1.01  &    2.15E-02&     1.08	 &    6.79E-02&     1.22  \\ \cline{2-8}
  \hline
\multirow{6}{0.6cm}{$\mathcal{P}_1$}& $   20$ &    4.34E-02&     --    &    1.81E-01&     --     &    5.85E-01&     --    \\ \cline{2-8}
& $   40$ &    2.01E-02&     1.11  &    1.05E-01&     0.78	 &    3.39E-01&     0.79  \\ \cline{2-8}
& $   80$ &    9.56E-03&     1.08  &    5.47E-02&     0.94	 &    1.75E-01&     0.95  \\ \cline{2-8}
& $  160$ &    4.64E-03&     1.04  &    2.76E-02&     0.99	 &    9.70E-02&     0.85  \\ \cline{2-8}
& $  320$ &    2.28E-03&     1.02  &    1.39E-02&     0.99	 &    4.22E-02&     1.20  \\ \cline{2-8}
& $  640$ &    1.13E-03&     1.01  &    6.95E-03&     1.00	 &    2.04E-02&     1.05  \\ \cline{2-8}
  \hline
\multirow{6}{0.6cm}{$\mathcal{P}_2$}& $   20$ &    1.14E-02&     --    &    1.38E-01&      --    &    4.45E-01&      --    \\ \cline{2-8}
& $   40$ &    2.09E-03&     2.45  &    4.81E-02&     1.52	 &    2.75E-01&     0.70   \\ \cline{2-8}
& $   80$ &    4.56E-04&     2.19  &    8.83E-03&     2.45	 &    1.52E-01&     0.86   \\ \cline{2-8}
& $  160$ &    1.13E-04&     2.02  &    1.31E-03&     2.76	 &    4.33E-02&     1.81   \\ \cline{2-8}
& $  320$ &    2.82E-05&     2.00  &    3.14E-04&     2.06	 &    4.13E-03&     3.39   \\ \cline{2-8}
& $  640$ &    7.04E-06&     2.00  &    7.81E-05&     2.01	 &    4.57E-04&     3.18   \\ \cline{2-8} \hline
\multirow{6}{0.6cm}{$\mathcal{P}_3$}& $   20$ &    1.84E-03&     --    &    5.50E-02&     --     &    3.63E-01&     --    \\ \cline{2-8}
& $   40$ &    2.01E-04&     3.19  &    7.63E-03&     2.85	 &    2.10E-01&     0.79  \\ \cline{2-8}
& $   80$ &    2.64E-05&     2.93  &    7.74E-04&     3.30	 &    5.68E-02&     1.89  \\ \cline{2-8}
& $  160$ &    3.38E-06&     2.97  &    9.62E-05&     3.01	 &    4.66E-03&     3.61  \\ \cline{2-8}
& $  320$ &    4.27E-07&     2.98  &    1.26E-05&     2.94	 &    1.98E-04&     4.56  \\ \cline{2-8}
& $  640$ &    6.33E-08&     2.75  &    1.60E-06&     2.97	 &    2.79E-05&     2.83  \\ \cline{2-8} \hline
\end{tabular}

	\caption{ Example \ref{example2}: Errors and orders of the conservative scheme at the terminal time $T=0.5$. }\label{tab_ex2_tab_err1}

\end{table}

\begin{table}[!ht]\centering
\begin{tabular}{|c|c|cc|cc|cc|}
  \hline
  & & \multicolumn{2}{c|}{$\alpha=0.5$}&\multicolumn{2}{c|}{$\alpha=0.7$}&\multicolumn{2}{c|}{$\alpha=0.8$}\\ \hline
  & $N$ & Error & Order & Error  & Order& Error &Order\\ \hline
  \multirow{6}{0.6cm}{$\mathcal{P}_0$}& $   20$ &    1.71E-01&     --    &    7.12E-01&      --    &    1.45E+00&     --   \\ \cline{2-8}
& $   40$ &    9.60E-02&     0.83  &    4.93E-01&     0.53	 &    1.16E+00&     0.32 \\ \cline{2-8}
& $   80$ &    5.06E-02&     0.92  &    3.04E-01&     0.70	 &    8.54E-01&     0.44 \\ \cline{2-8}
& $  160$ &    2.59E-02&     0.97  &    1.70E-01&     0.84	 &    5.61E-01&     0.61 \\ \cline{2-8}
& $  320$ &    1.31E-02&     0.98  &    8.95E-02&     0.92	 &    3.31E-01&     0.76 \\ \cline{2-8}
& $  640$ &    6.59E-03&     0.99  &    4.58E-02&     0.97	 &    1.79E-01&     0.88 \\ \cline{2-8} \hline
\multirow{6}{0.6cm}{$\mathcal{P}_1$}& $   20$ &    3.85E-02&      --   &    2.08E-01&     --    &    6.48E-01&      --   \\ \cline{2-8}
& $   40$ &    1.73E-02&     1.16  &    8.42E-02&     1.31	&    2.93E-01&     1.15	 \\ \cline{2-8}
& $   80$ &    8.51E-03&     1.02  &    4.23E-02&     0.99	&    1.42E-01&     1.04	 \\ \cline{2-8}
& $  160$ &    4.25E-03&     1.00  &    2.00E-02&     1.08	&    7.10E-02&     1.00	 \\ \cline{2-8}
& $  320$ &    2.08E-03&     1.03  &    8.25E-03&     1.28	&    2.10E-02&     1.75	 \\ \cline{2-8}
& $  640$ &    9.81E-04&     1.08  &    2.97E-03&     1.48	&    7.00E-03&     1.59	 \\ \cline{2-8}
  \hline
\multirow{6}{0.6cm}{$\mathcal{P}_2$}& $   20$ &    9.63E-03&      --  &    1.18E-01&       --   &    3.67E-01&      --   \\ \cline{2-8}
& $   40$ &    1.71E-03&     2.49 &    3.97E-02&     1.58	&    2.14E-01&     0.78	 \\ \cline{2-8}
& $   80$ &    4.21E-04&     2.03 &    7.00E-03&     2.50	&    1.06E-01&     1.01	 \\ \cline{2-8}
& $  160$ &    1.08E-04&     1.97 &    1.01E-03&     2.79	&    2.53E-02&     2.07	 \\ \cline{2-8}
& $  320$ &    2.69E-05&     2.00 &    2.16E-04&     2.23	&    1.21E-03&     4.39	 \\ \cline{2-8}
& $  640$ &    6.59E-06&     2.03 &    4.14E-05&     2.38	&    1.16E-04&     3.38	 \\ \cline{2-8} \hline
\multirow{6}{0.6cm}{$\mathcal{P}_3$}& $   20$ &    1.48E-03&      --  &    4.82E-02&      --    &    2.95E-01&      --   \\ \cline{2-8}
& $   40$ &    1.82E-04&     3.02 &    6.23E-03&     2.95	&    1.57E-01&     0.91	 \\ \cline{2-8}
& $   80$ &    2.49E-05&     2.87 &    5.97E-04&     3.38	&    4.21E-02&     1.90	 \\ \cline{2-8}
& $  160$ &    3.17E-06&     2.98 &    6.25E-05&     3.26	&    1.88E-03&     4.48	 \\ \cline{2-8}
& $  320$ &    3.90E-07&     3.02 &    5.85E-06&     3.42	&    6.23E-05&     4.92	 \\ \cline{2-8}
& $  640$ &    5.67E-08&     2.78 &    4.05E-07&     3.85	&    6.35E-06&     3.29	 \\ \cline{2-8} \hline
\end{tabular}

\caption{ Example \ref{example2}: Errors and orders of the dissipative scheme at the terminal time $T=0.5$. }\label{tab_ex2_tab_err2}

\end{table}

\begin{exmp}\label{example_kappa_1}
We consider the one-soliton solution of \eqref{eq:mch} with $\kappa\neq0$ in \cite{MatsunoJPA2014}. The parametric representation of the one-soliton solution reads as
\begin{align*}
U=\frac{4\kappa^2 \alpha}{(1-(\kappa \alpha)^2)^{3/2}}\frac{\cosh \xi}{\cosh 2\xi +\frac{1+(\kappa \alpha)^2}{1-(\kappa \alpha)^2}},
\end{align*}
where
\begin{align*}
x-ct-x_0=\frac{\xi}{\kappa \alpha}+ \ln \frac{1-\kappa \alpha \tanh \xi}{1+\kappa \alpha \tanh \xi}
\end{align*}
with
\begin{align*}
c=\frac{2\kappa ^2}{1-(\kappa \alpha)^2}
\end{align*}
\end{exmp}

The second experiment investigates a one-soliton solution with nonzero $\kappa$ and provides an additional validation of the error analysis. We set $x_0=0$, $\kappa =1$, and select
$\alpha=0.3,0.5,0.7$, for which the smoothness condition $0<\kappa \alpha<1/\sqrt{2}(\simeq0.707)$ is satisfied. The
computation is performed on $\Omega=(-20,20)$ with compact-support boundary conditions
and final time $T=5.0$. The exact wave profiles are plotted in Figure \ref{fig_ex2_1} (right). The
convergence data in Tables \ref{tab_exkappa1_tab_err1} and \ref{tab_exkappa1_tab_err2} show that both conservative and dissipative schemes
reproduce the theoretical $k$-th order convergence behavior.

\begin{table}[!ht]\centering
\begin{tabular}{|c|c|cc|cc|c|cc|}
  \hline
  & & \multicolumn{2}{c|}{$\alpha=0.3$}&\multicolumn{2}{c|}{$\alpha=0.5$}&\multicolumn{3}{c|}{$\alpha=0.7$}\\ \hline
  & $N$ & Error & Order & Error  & Order& $N$ & Error &Order\\ \hline
\multirow{6}{0.6cm}{$\mathcal{P}_1$}& $  160$ &    3.07E-02&      --   &    1.45E-01&      --   & $   160$ &    7.16E-01&	  --   \\ \cline{2-9}
									& $  320$ &    1.58E-02&     0.96  &    7.64E-02&     0.95	& $   320$ &    5.76E-01&     0.31 \\ \cline{2-9}
									& $  640$ &    8.07E-03&     0.97  &    3.90E-02&     0.97	& $   640$ &    4.54E-01&     0.35 \\ \cline{2-9}
									& $ 1280$ &    4.09E-03&     0.98  &    1.97E-02&     0.98	& $  1280$ &    3.20E-01&     0.51 \\ \cline{2-9}
									& $ 2560$ &    2.06E-03&     0.99  &    9.94E-03&     0.99	& $  2560$ &    2.11E-01&     0.60 \\ \cline{2-9}
									& $ 5120$ &    1.03E-03&     0.99  &    4.98E-03&     0.99	& $  5120$ &    1.28E-01&     0.72 \\ \cline{2-9}
  \hline
\multirow{6}{0.6cm}{$\mathcal{P}_2$}& $  160$ &    1.15E-03&	   -- &    2.35E-02&	  --    & $  5120$&   8.80E-02&   --   \\ \cline{2-9}
									& $  320$ &    2.91E-04&     1.99 &    3.43E-03&     2.78	& $  10240$&  1.73E-02&     2.35   	 \\ \cline{2-9}
									& $  640$ &    7.29E-05&     2.00 &    8.08E-04&     2.09	& $  20480$&  6.83E-04&     4.66    	 \\ \cline{2-9}
									& $ 1280$ &    1.83E-05&     2.00 &    2.00E-04&     2.02	& $  40960$&  1.46E-04 &    2.23	 \\ \cline{2-9}
									& $ 2560$ &    4.58E-06&     2.00 &    4.98E-05&     2.00	& $  81920$&  3.40E-05 &    2.10 	 \\ \cline{2-9}
									& $ 5120$ &    1.15E-06&     2.00 &    1.24E-05&     2.01	& $  163840$& 8.43E-06 &    2.01	 \\ \cline{2-9} \hline

\multirow{6}{0.6cm}{$\mathcal{P}_3$}& $  160$ &    5.67E-05&	  --  &    1.72E-03&      --    & $  640$ &    4.78E-01&      --  \\ \cline{2-9}
									& $  320$ &    7.60E-06&     2.90 &    2.37E-04&     2.85	& $ 1280$ &    1.47E-01&     1.70	 \\ \cline{2-9}
									& $  640$ &    9.73E-07&     2.97 &    3.03E-05&     2.97	& $ 2560$ &    2.02E-02&     2.87	 \\ \cline{2-9}
									& $ 1280$ &    1.22E-07&     2.99 &    3.99E-06&     2.92	& $ 5120$ &    2.48E-03&     3.02	 \\ \cline{2-9}
									& $ 2560$ &    1.52E-08&     3.00 &    5.10E-07&     2.97	& $10240$ &    3.49E-04&     2.83	 \\ \cline{2-9}
									& $ 5120$ &    1.90E-09&     3.00 &    6.41E-08&     2.99	& $20480$ &    6.22E-05&     2.49	 \\ \cline{2-9} \hline
\end{tabular}

\caption{ Example \ref{example_kappa_1}: Errors and orders of the conservative scheme at the terminal time $T=5.0$. }\label{tab_exkappa1_tab_err1}
\end{table}

\begin{table}[!ht]\centering
\begin{tabular}{|c|c|cc|cc|c|cc|}
  \hline
  & & \multicolumn{2}{c|}{$\alpha=0.3$}&\multicolumn{2}{c|}{$\alpha=0.5$}&\multicolumn{3}{c|}{$\alpha=0.7$}\\ \hline
  & $N$ & Error & Order & Error  & Order& $N$ & Error &Order\\ \hline
\multirow{6}{0.6cm}{$\mathcal{P}_1$}& $  160$ &    3.09E-02&      --   &    1.47E-01&      --   & $   160$ &    7.17E-01&	  --   \\ \cline{2-9}
									& $  320$ &    1.58E-02&     0.96  &    7.64E-02&     0.95	& $   320$ &    5.76E-01&     0.31 \\ \cline{2-9}
									& $  640$ &    8.07E-03&     0.97  &    3.90E-02&     0.97	& $   640$ &    4.54E-01&     0.35 \\ \cline{2-9}
									& $ 1280$ &    4.09E-03&     0.98  &    1.97E-02&     0.98	& $  1280$ &    3.20E-01&     0.51 \\ \cline{2-9}
									& $ 2560$ &    2.06E-03&     0.99  &    9.94E-03&     0.99	& $  2560$ &    2.11E-01&     0.60 \\ \cline{2-9}
									& $ 5120$ &    1.03E-03&     0.99  &    4.98E-03&     0.99	& $  5120$ &    1.28E-01&     0.72 \\ \cline{2-9}
  \hline
\multirow{6}{0.6cm}{$\mathcal{P}_2$}& $  160$ &    1.15E-03&	   -- &    2.35E-02&	  --    & $  5120$&   8.80E-02&   --   \\ \cline{2-9}
									& $  320$ &    2.91E-04&     1.99 &    3.43E-03&     2.78	& $  10240$&  1.73E-02&     2.35   	 \\ \cline{2-9}
									& $  640$ &    7.29E-05&     2.00 &    8.08E-04&     2.09	& $  20480$&  6.83E-04&     4.66    	 \\ \cline{2-9}
									& $ 1280$ &    1.83E-05&     2.00 &    2.00E-04&     2.02	& $  40960$&  1.46E-04 &    2.23	 \\ \cline{2-9}
									& $ 2560$ &    4.58E-06&     2.00 &    4.98E-05&     2.00	& $  81920$&  3.40E-05 &    2.10 	 \\ \cline{2-9}
									& $ 5120$ &    1.15E-06&     2.00 &    1.24E-05&     2.01	& $  163840$& 8.43E-06 &    2.01	 \\ \cline{2-9} \hline

\multirow{6}{0.6cm}{$\mathcal{P}_3$}& $  160$ &    5.67E-05&	  --  &    1.72E-03&      --    & $  640$ &    4.78E-01&      --  \\ \cline{2-9}
									& $  320$ &    7.60E-06&     2.90 &    2.37E-04&     2.85	& $ 1280$ &    1.47E-01&     1.70	 \\ \cline{2-9}
									& $  640$ &    9.73E-07&     2.97 &    3.03E-05&     2.97	& $ 2560$ &    2.02E-02&     2.87	 \\ \cline{2-9}
									& $ 1280$ &    1.22E-07&     2.99 &    3.99E-06&     2.92	& $ 5120$ &    2.48E-03&     3.02	 \\ \cline{2-9}
									& $ 2560$ &    1.52E-08&     3.00 &    5.10E-07&     2.97	& $10240$ &    3.49E-04&     2.83	 \\ \cline{2-9}
									& $ 5120$ &    1.90E-09&     3.00 &    6.41E-08&     2.99	& $20480$ &    6.22E-05&     2.49	 \\ \cline{2-9} \hline
\end{tabular}

\caption{ Example \ref{example_kappa_1}: Errors and orders of the dissipative scheme at the terminal time $T=5.0$. }\label{tab_exkappa1_tab_err2}

\end{table}
\begin{exmp}\label{example_kappa_2}
We consider the two-soliton solution of \eqref{eq:mch} with $\kappa\neq0$ in \cite{MatsunoJPA2014}. The parametric representation reads as follows
\begin{align*}
&u(y,t)=\frac{1}{2{\rm i} \kappa}\left(\ln\frac{\bar{f}\bar{g}}{fg}\right)_t,\\
&x(y,t)=\frac{y}{\kappa}+\ln \frac{\bar{g}{g}}{\bar{f}f},
\end{align*}
where
\begin{align*}
&f=1+{\rm i}(e^{\xi_1+\psi_1}+e^{\xi_2+\psi_2})-\left(\frac{\alpha_1-\alpha_2}{\alpha_1+\alpha_2}\right)^2e^{\xi_1+\xi_2+\psi_1+\psi_2},\\
&g=1+{\rm i}(e^{\xi_1-\psi_1}+e^{\xi_2-\psi_2})-\left(\frac{\alpha_1-\alpha_2}{\alpha_1+\alpha_2}\right)^2e^{\xi_1+\xi_2-\psi_1-\psi_2},\\
&\xi_j=\alpha_j\left(y-\frac{2\kappa^3}{1-(\kappa \alpha_j)^2}t\right),\quad \psi_j=\ln\sqrt{\frac{1+\kappa \alpha_j}{1-\kappa \alpha_j}},\quad j=1,2.
\end{align*}
\end{exmp}

We test this example on uniform meshes for $\kappa=1$, $\alpha_1=0.7$ and $\alpha_2=0.5$, which ensures the smoothness of the solution. The computational domain is taken as $\Omega=(-100,100)$ with compact support boundary conditions and the terminal time $T=5$. We show the exact solution at different times in Figure \ref{fig_ex2_kappa_1}. For $k=1,2,3$, we also observe the $k$-th order convergence rates in Table \ref{tab_exkappa2_tab_err2}. 

\begin{figure}
  \centering
  \includegraphics[width=0.45\textwidth]{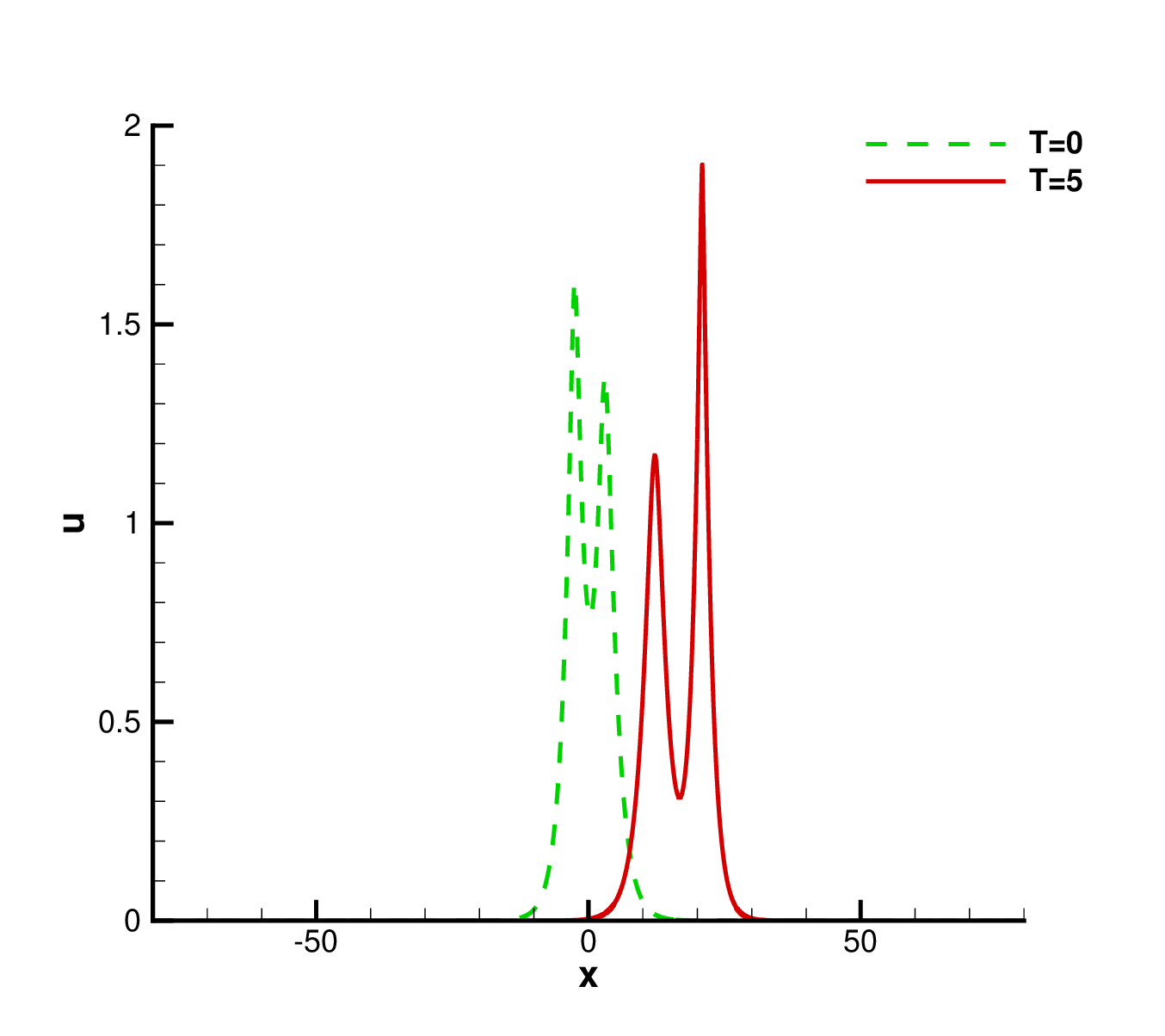}
  \caption{{ Example \ref{example_kappa_2}: The exact solution at time $T=0$ and $T=5$ }}\label{fig_ex2_kappa_1}
\end{figure}

\begin{table}[!ht]\centering
\begin{tabular}{|c|c|cc|cc|}
  \hline
  & & \multicolumn{2}{c|}{The conservative scheme}&\multicolumn{2}{c|}{The dissipative scheme}\\ \hline
  & $N$ & Error & Order & Error  & Order\\ \hline
\multirow{6}{0.6cm}{$\mathcal{P}_1$}& $ 1000$ &    2.95E-01&	   --   &    2.93E-01&      --       \\ \cline{2-6}
									& $ 2000$ &    1.81E-01&     0.70   &    1.81E-01&     0.70   	 \\ \cline{2-6}
									& $ 4000$ &    9.96E-02&     0.86   &    9.95E-02&     0.86   	 \\ \cline{2-6}
									& $ 8000$ &    5.60E-02&     0.83   &    5.60E-02&     0.83   	 \\ \cline{2-6}
									& $16000$ &    1.86E-02&     1.59   &    1.86E-02&     1.59   	 \\ \cline{2-6}
									& $32000$ &    9.20E-03&     1.02   &    9.20E-03&     1.02   	 \\ \cline{2-6}
  \hline
\multirow{6}{0.6cm}{$\mathcal{P}_2$}& $ 1000$ &    3.09E-01&	  --   &     3.06E-01&     --    \\ \cline{2-6}
									& $ 2000$ &    1.06E-01&     1.54  &     1.06E-01&     1.54	 \\ \cline{2-6}
									& $ 4000$ &    2.97E-02&     1.83  &     2.97E-02&     1.83	 \\ \cline{2-6}
									& $ 8000$ &    2.55E-03&     3.54  &     2.55E-03&     3.54	 \\ \cline{2-6}
									& $16000$ &    3.58E-04&     2.83  &     3.58E-04&     2.83	 \\ \cline{2-6}
									& $32000$ &    8.66E-05&     2.05  &     8.66E-05&     2.05	 \\ \cline{2-6} \hline

\multirow{6}{0.6cm}{$\mathcal{P}_3$}& $ 1000$ &    7.49E-02&      --   &    7.49E-02&      --    \\ \cline{2-6}
									& $ 2000$ &    1.77E-02&     2.08  &    1.77E-02&     2.08	 \\ \cline{2-6}
									& $ 4000$ &    2.17E-03&     3.02  &    2.17E-03&     3.02	 \\ \cline{2-6}
									& $ 8000$ &    5.55E-05&     5.29  &    5.55E-05&     5.29	 \\ \cline{2-6}
									& $16000$ &    2.82E-06&     4.30  &    2.82E-06&     4.30	 \\ \cline{2-6}
									& $32000$ &    3.68E-07&	 2.94  &    3.68E-07&	  2.94   \\ \cline{2-6} \hline
\end{tabular}

\caption{ Example \ref{example_kappa_2}: Errors and orders of the conservative and dissipative schemes at the terminal time $T=5.0$. }\label{tab_exkappa2_tab_err2}

\end{table}

\begin{exmp}\label{example3}
We consider the one-peakon solution in \cite{GuiCIMP2013} $U(x,t)=\sqrt{\frac{3c}{2}}e^{-|x-ct|}$ with $c=1.0$ of the mCH equation \eqref{eq:mch} with $\kappa=0$, which is a right-going traveling wave solution and a single peakon solution. It is noted that the $H^1$-norm is conserved for the one-peakon solution \cite {chang2016lax,chang2017liouville,Chang2018CMP}. 
\end{exmp}

The computational domain is $\Omega=(-20,20)$ with compact support boundary conditions, and the terminal time is $T=5.0$. We plot the profiles of the numerical solutions at different times for $k=3$ on a uniform mesh with $h=$ 6.25E-02, see Figure \ref{fig_ex3_2}. We observe that the conservative LDG method can capture the peakon structure better than the dissipative LDG method for long-time evolution. We also show the differences $E(0)-E(t)$ in Figure \ref{fig_ex3_2}; it shows that the conservative scheme maintains the energy difference near machine precision, whereas the dissipative scheme dissipates energy monotonically.

\begin{figure}
  \centering
  \includegraphics[width=0.3\textwidth]{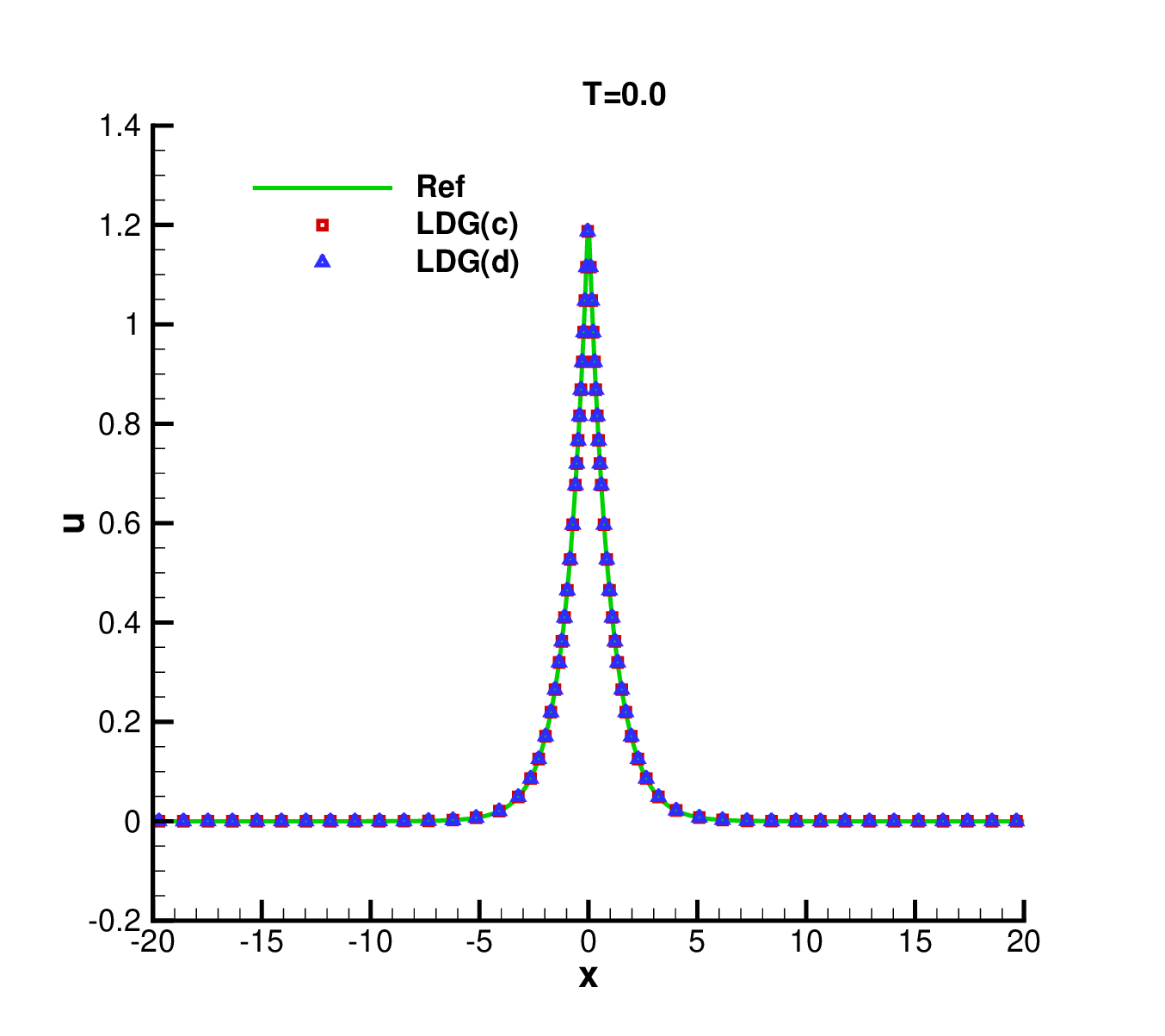}
  \includegraphics[width=0.3\textwidth]{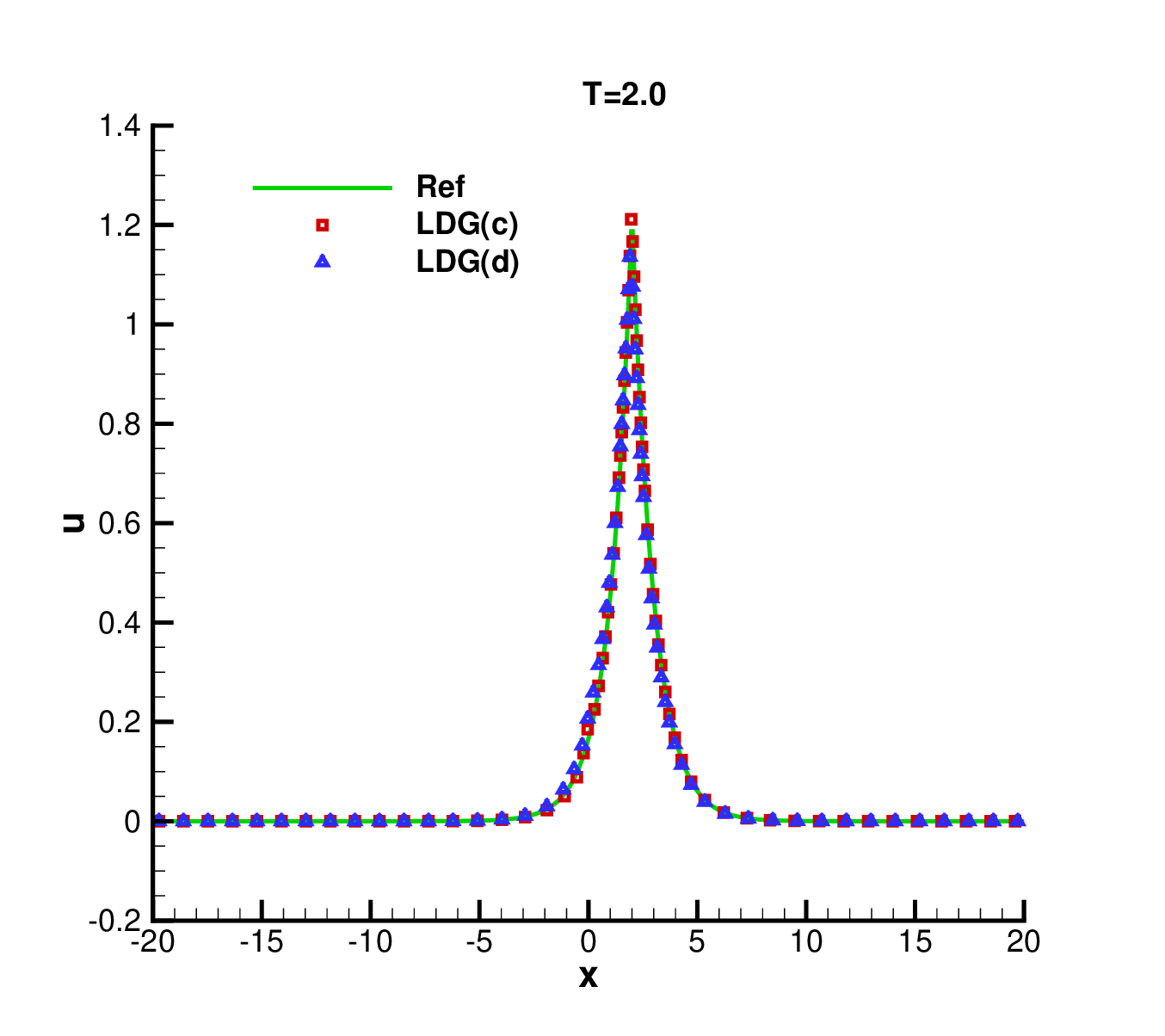}
  \includegraphics[width=0.3\textwidth]{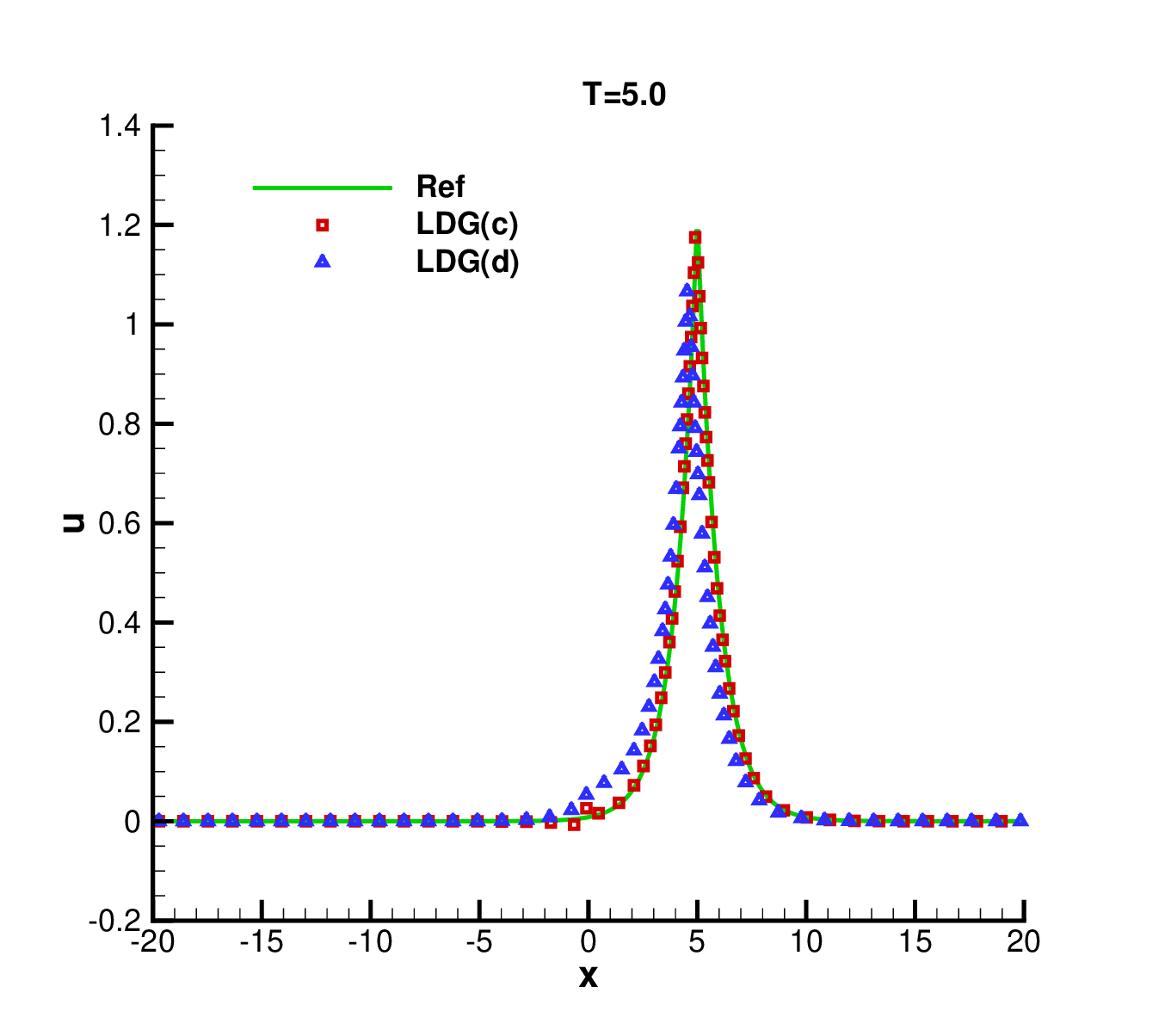}
  \includegraphics[width=0.45\textwidth]{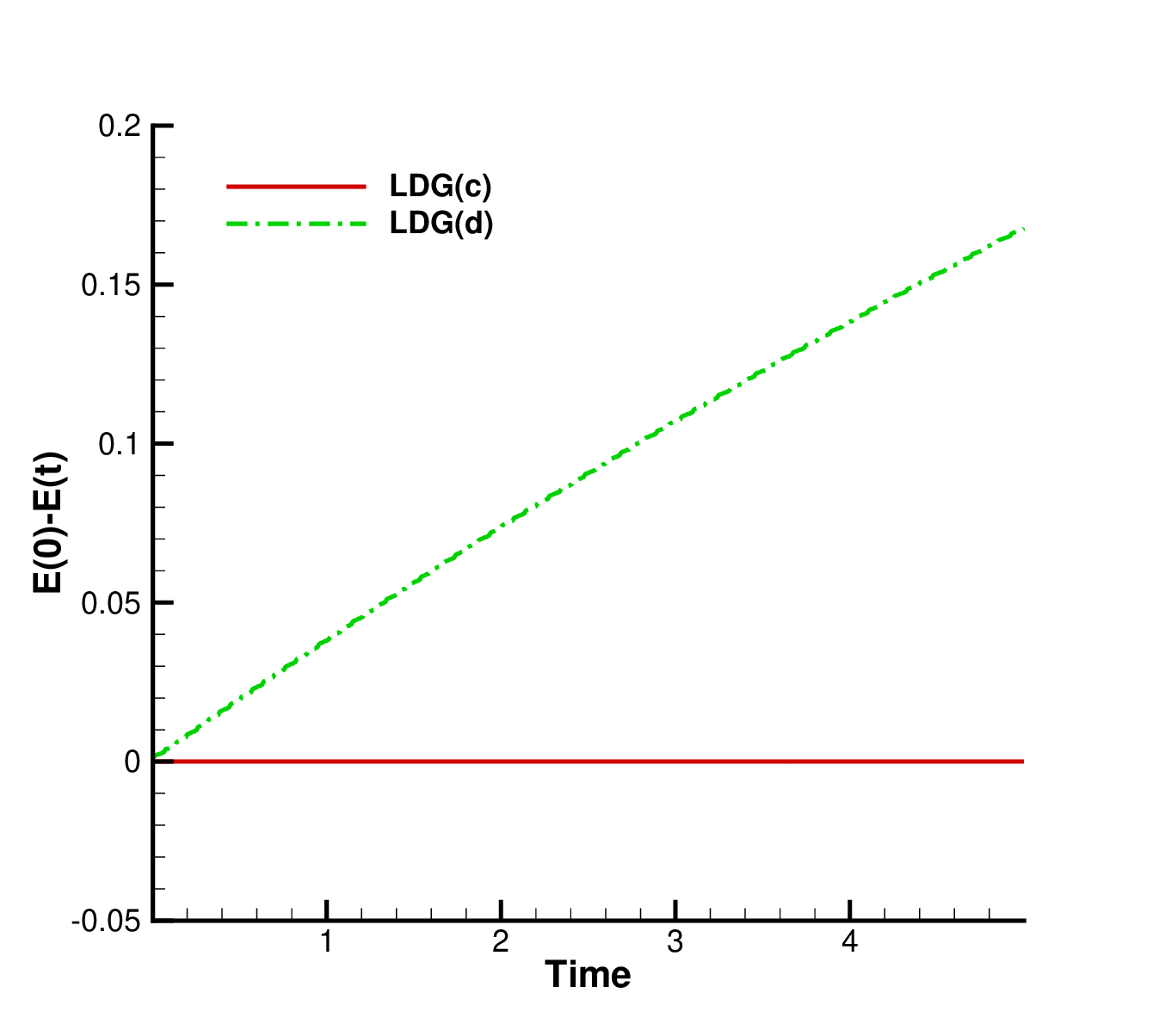}
  \includegraphics[width=0.45\textwidth]{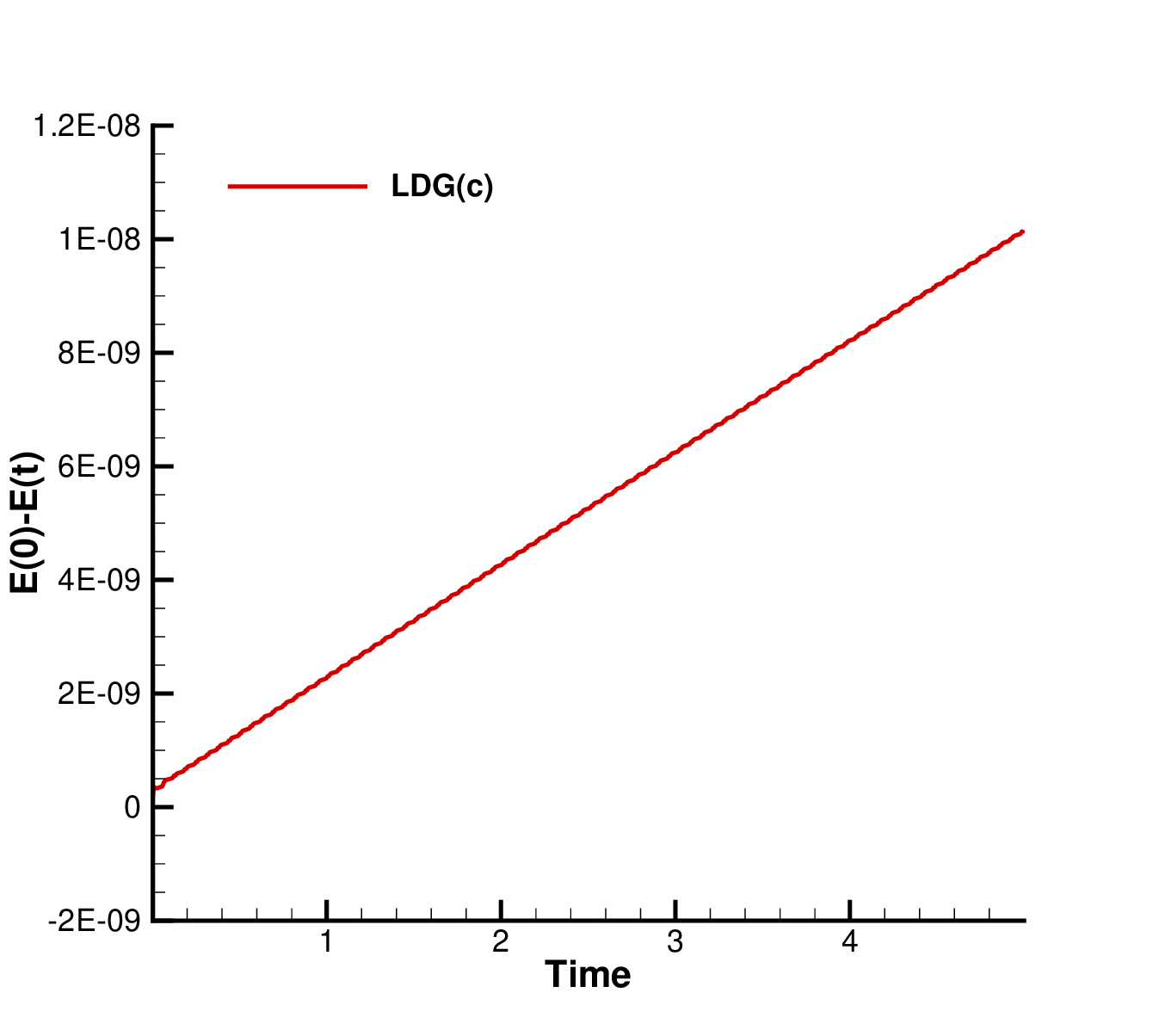}
  \caption{{ Example \ref{example3}: The numerical solutions $u$ for $k=3$ on a uniform mesh with $h=$ 6.25E-02.}}\label{fig_ex3_2}
\end{figure}

\begin{exmp}\label{example4}
Next, we consider the periodic peakon solution of the modified Camassa-Holm equation with $\kappa=0$ $$U(x,t)=\sqrt{3c/(2\cosh(\pi)^2+1)}\cosh (x-ct-2\pi \lfloor (x-ct)/(2\pi)\rfloor-\pi)$$ in \cite{QuCMP2013}. The periodic domain is $\Omega=(-3\pi,3\pi)$ and the terminal time is $T=5.0$.
\end{exmp}

We solve this example by using the LDG method for $k=3$ on a uniform mesh with $N=1280$. The profiles of the numerical solutions $u$ at different times are shown in Figure \ref{fig_ex4_2}. The energy difference $E(0)-E(t)$ versus time is also plotted in Figure \ref{fig_ex4_2}. We again observe energy conservation and energy stability for the conservative and dissipative schemes, respectively.

\begin{figure}
  \centering
  \includegraphics[width=0.45\textwidth]{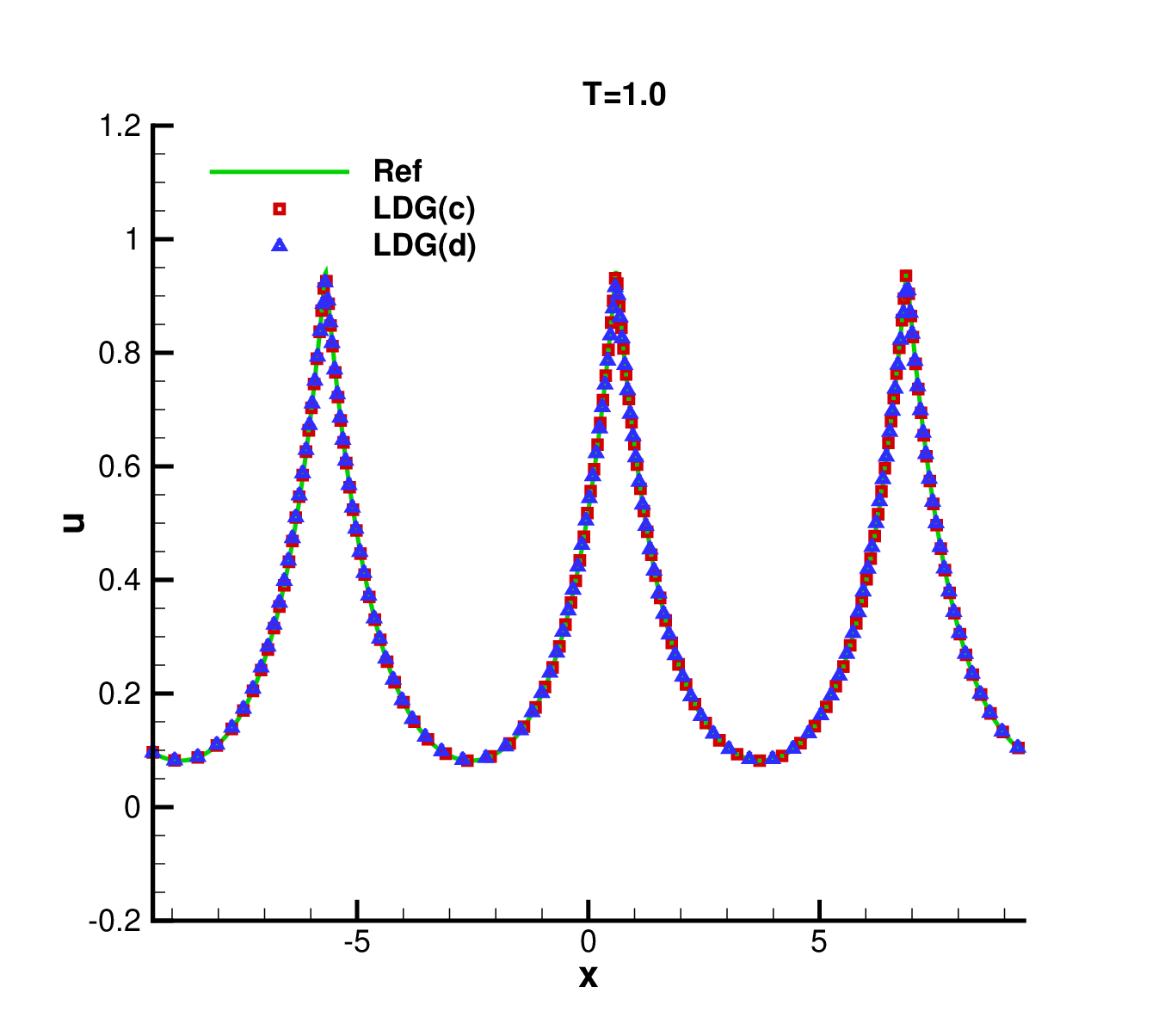}
  \includegraphics[width=0.45\textwidth]{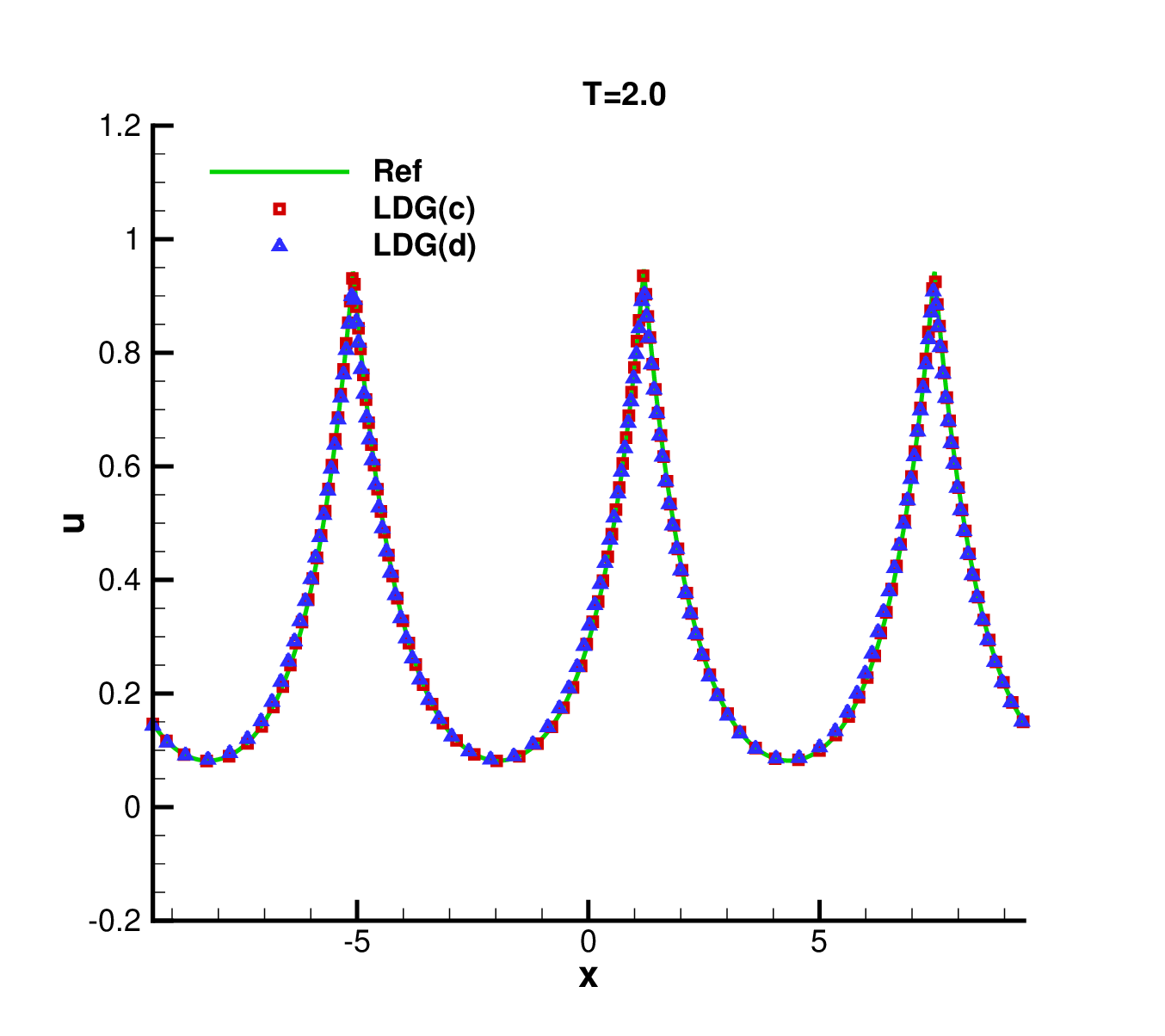}
  \includegraphics[width=0.45\textwidth]{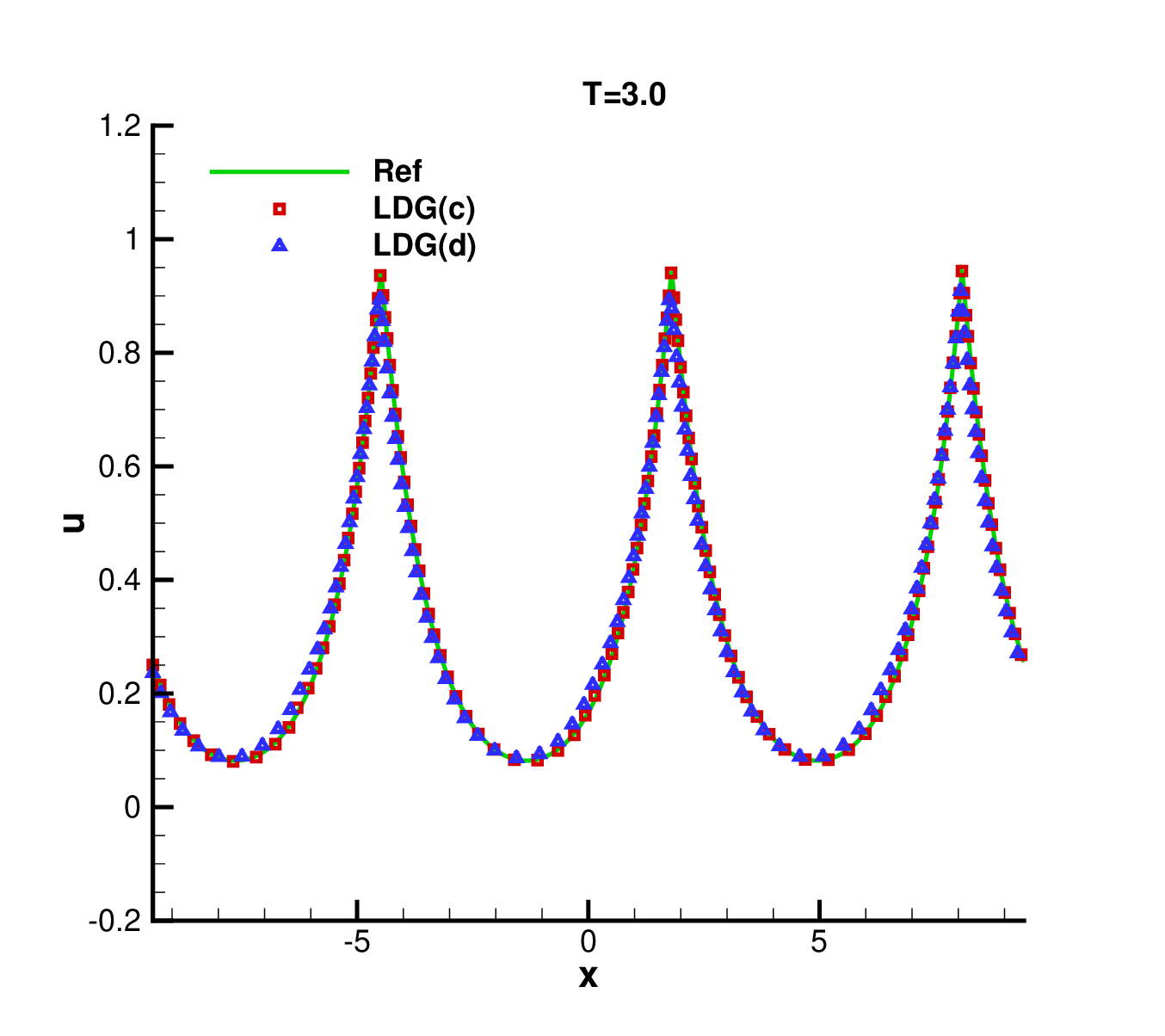}
  \includegraphics[width=0.45\textwidth]{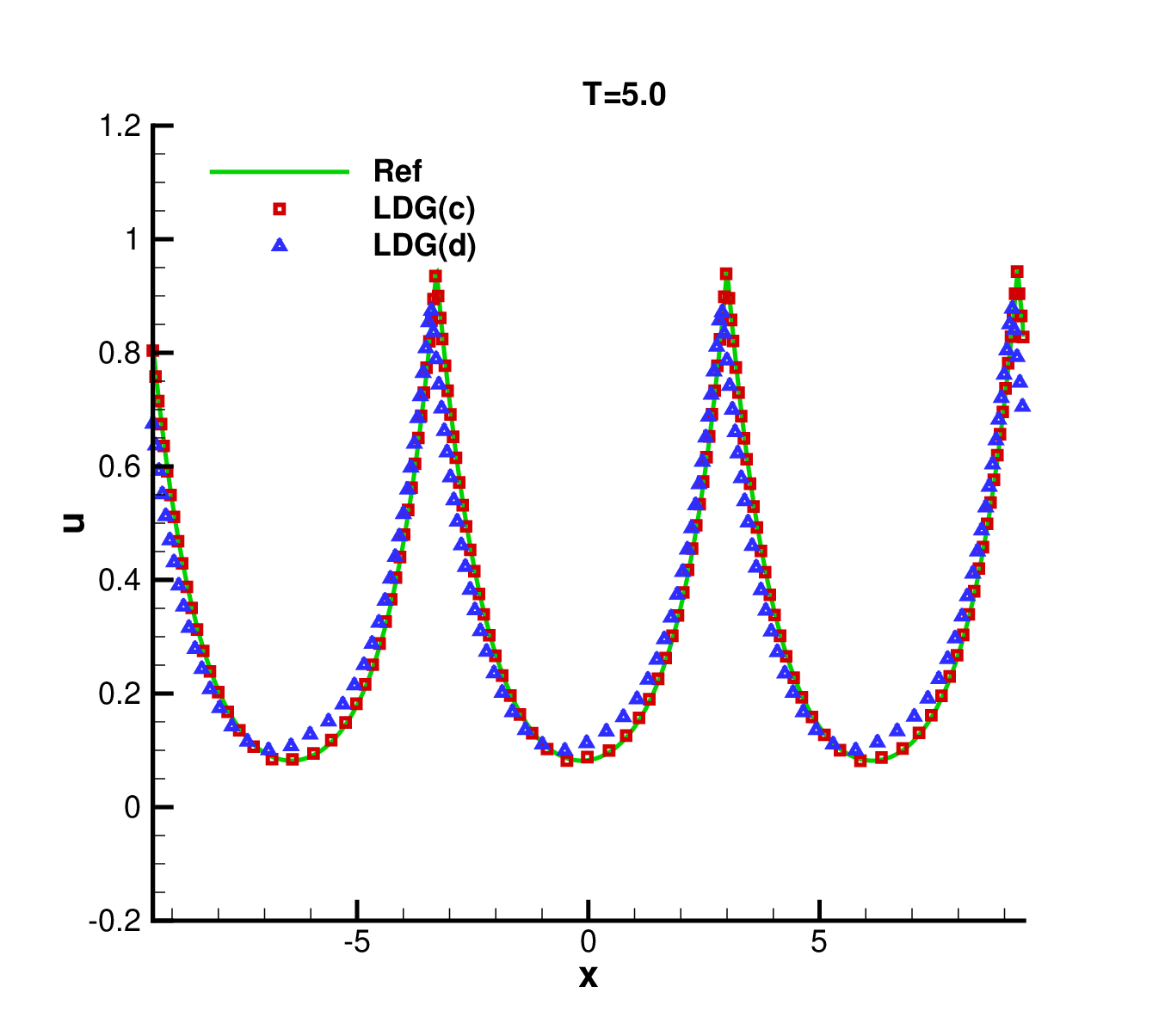}
  \includegraphics[width=0.45\textwidth]{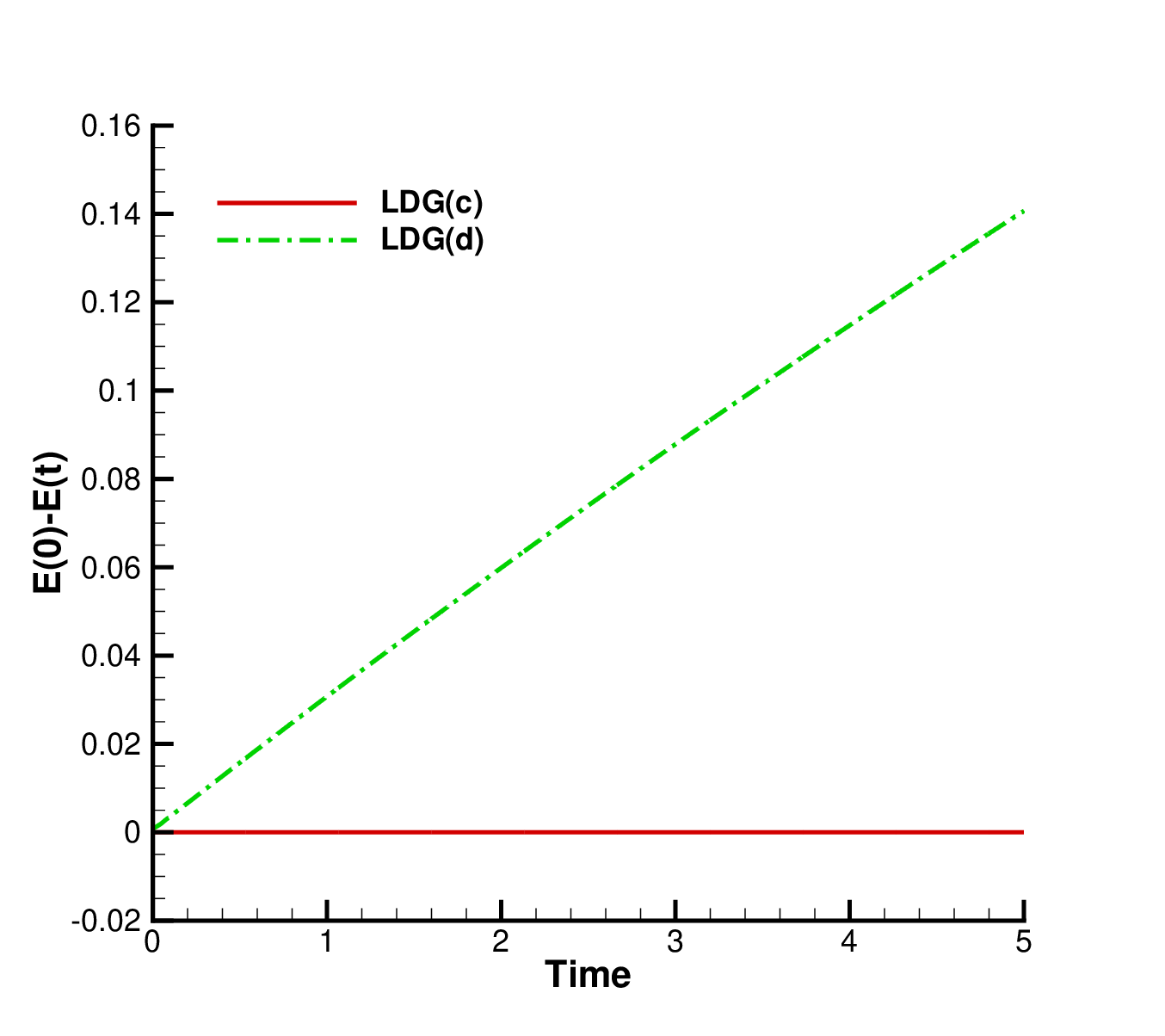}
  \includegraphics[width=0.45\textwidth]{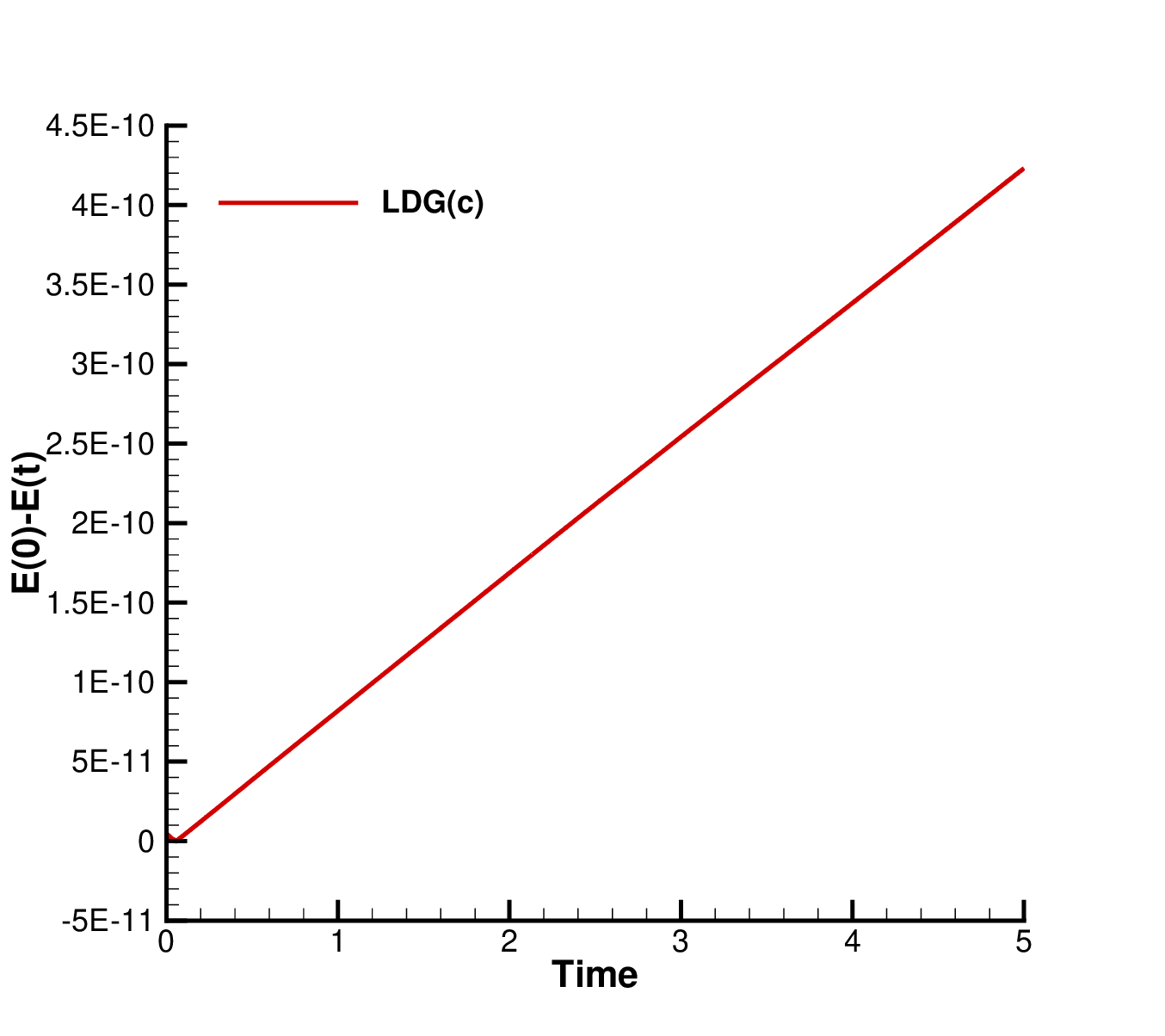}
  \caption{{ Example \ref{example4}: The numerical solutions $u$ for $k=3$ on a uniform mesh with $N=1280$.}}\label{fig_ex4_2}
\end{figure}

\section{Concluding remarks}
\label{sec6}

In this work, we propose an LDG method for solving the mCH equation. For general solutions, we prove energy stability; for smooth solutions, we obtain an a priori error estimate. The nonlinear stability enables us to handle the nonlinear spatial discretization terms and to derive the optimal error estimate for $k>\frac{1}{2}$ together with an a priori assumption on $r$. Our numerical results confirm that the proposed schemes are of arbitrarily high order and can effectively capture peakon solutions. The analysis presented here concerns only the semi-discrete scheme; the fully discrete version will be studied in our subsequent work.

\section*{Acknowledgement} 
We thank Professor Chi-Wang Shu for his valuable communications.


\end{document}